 \RequirePackage{fix-cm}

 \documentclass[smallextended]{svjour3} 
 \smartqed  

\usepackage{amssymb,amsmath,amsfonts,latexsym} 
\usepackage{hyperref}
   
\usepackage{mathtools}
\usepackage{multirow,array}
\usepackage{graphicx}
\usepackage{subfigure}
\usepackage{epstopdf}
\usepackage{float} 
\usepackage{color, xcolor}
\allowdisplaybreaks

\newcommand{\ep}{\varepsilon}
\newcommand{\la}{\lambda}
\newcommand{\va}{\varphi}
\newcommand{\ppp}{\partial}

\newcommand{\www}{\widetilde}

\newcommand{\pppa}{\partial_t^{\alpha}}

\newcommand{\ddda}{d_t^{\alpha}}

\newcommand{\HOLST}{C^{2+\theta}(\ooo{\OOO})}
\newcommand{\HOLSZ}{C^{\theta}(\ooo{\OOO})}
\newcommand{\CONE}{C^1(\ooo{\OOO})}

\newcommand{\Ahalf}{A_0^{\frac{1}{2}}}
\newcommand{\QQQQ}{\Omega \times (0,T)}

\newcommand{\R}{\mathbb{R}}

\newcommand{\N}{\mathbb{N}}

\newcommand{\ooo}{\overline}
\newcommand{\OOO}{\Omega}
\newcommand{\MLONE}{E_{\alpha,1}}
\newcommand{\MLTWO}{E_{\alpha,\alpha}}

\newcommand{\sumij}{\sum_{i,j=1}^d}

\newcommand{\NUNU}{\ppp_{\nu_A}}

\newcommand{\CC}{{_{0}C^1[0,T]}}

\newcommand{\HH}{H_{\alpha}}

\newcommand{\hhalf}{\frac{1}{2}}
\newcommand{\DDD}{\mathcal{D}}

\newcommand{\sumn}{\sum_{n=1}^{\infty}}

\journalname{Fract. Calc. Appl. Anal.} 

\begin{document}


\title{Comparison principles for the time-fractional diffusion equations with the Robin boundary conditions. Part I: Linear equations}

\titlerunning{Comparison principles for the  time-fractional diffusion equations \dots}

\author{
        Yuri Luchko$^1$ 
\and
        Masahiro Yamamoto$^2$ 
 }

\authorrunning{Yu. Luchko \and M. Yamamoto} 

\institute{Yuri Luchko$^{1,*}$
\at
Department of Mathematics, Physics, and Chemistry, 
Berlin University of Applied Sciences and Technology,
Luxemburger Str. 10, 
Berlin -- 13353, Germany \\
\email{luchko@bht-berlin.de} $^*$ corresponding author 
 \and
Masahiro Yamamto$^{2}$
\at
Department of Mathematical Sciences, 
The University of Tokyo,
Komaba, Meguro, 
Tokyo -- 153, Japan \\
Honorary Member of Academy of Romanian Scientists,
Ilfov, nr. 3, Bucuresti, Romania \\
Correspondence member of Accademia Peloritana dei Pericolanti,
Palazzo Universit\`a, Piazza S. Pugliatti 1,
Messina -- 98122, Italy \\
\email{myama@ms.u-tokyo.ac.jp}
}

\date{Received: XX 2023 / Revised: ....  / Accepted: ......}


\maketitle

\begin{abstract}
{The main objective of this paper is analysis of the initial-boundary value 
problems for the linear  time-fractional diffusion equations
with a uniformly elliptic spatial differential operator of the second order and
the Caputo type time-fractional derivative  acting in the fractional Sobolev 
spaces.
The boundary conditions are formulated in form of the homogeneous
Neumann or Robin conditions. First we discuss the uniqueness and existence of 
solutions to these initial-boundary value problems. 
Under some suitable conditions on the problem data, we then prove positivity 
of the solutions. Based on these results,  several comparison principles for the solutions to the initial-boundary value problems for the linear  time-fractional diffusion equations are derived.}



\end{abstract} 


\section{Introduction} \label{sec:1}

\setcounter{section}{1}
\setcounter{equation}{0} 

In this paper, we deal with a linear time-fractional 
diffusion equation in the form
$$
 \pppa (u(x,t)-a(x)) = \sum_{i,j=1}^d \ppp_i(a_{ij}(x)\ppp_j u(x,t))
$$
\begin{equation}
\label{(1.1)}
+ \sum_{j=1}^d b_j(x,t)\ppp_ju(x,t) + c(x,t)u(x,t) + F(x,t),\quad 
 x \in \Omega,\, 0<t<T,             
\end{equation}
where $\pppa$ is the Caputo fractional derivative of order 
$\alpha\in (0,1)$ defined on the fractional Sobolev spaces 
(see Section \ref{sec2} for the details) and $\OOO \subset \R^d, \ d=1,2,3$ is
a bounded domain with a smooth boundary $\ppp\OOO$.   
All the functions under consideration are supposed to be real-valued.

In what follows, we always assume that the following conditions are satisfied:
\begin{equation}
\label{(1.2)}
\left\{ \begin{array}{rl}
& a_{ij} = a_{ji} \in C^1(\ooo{\OOO}), \quad 1\le i,j \le d, \\
& b_j,\, c \in C^1([0,T]; C^1(\ooo{\OOO})) \cap C([0,T];C^2(\ooo{\OOO})),
  \quad 1\le j \le d, \\
& \mbox{and there exists a constant $\kappa>0$ such that}\\
& \sumij a_{ij}(x)\xi_i\xi_j \ge \kappa \sum_{j=1}^d \xi_j^2, \quad
x\in \OOO, \, \xi_1, ..., \xi_d \in \R.
\end{array}\right.
\end{equation}

Using the notations $\ppp_j = \frac{\ppp}{\ppp x_j}$, $j=1, 2, ..., d$, 
we define a conormal derivative
$\NUNU w$ with respect to the differential operator 
$\sumij \ppp_j(a_{ij}\ppp_i)$ by 
\begin{equation}\label{(1.3)}
\NUNU w(x) =  \sumij a_{ij}(x)\ppp_jw(x)\nu_i(x), \quad x\in \ppp\OOO,
\end{equation}
where $\nu = \nu(x) =: (\nu_1(x), ..., \nu_d(x))$ is the unit outward normal
vector to $\ppp\OOO$ at the point $x := (x_1,..., x_d) \in \ppp\OOO$. 

For the equation \eqref{(1.1)},  we consider the initial-boundary value 
problems with the homogeneous Neumann boundary condition
\begin{equation}
\label{(1.3a)}
\ppp_{\nu_A}u = 0 \quad \mbox{on $\ppp\OOO \times (0,T)$}   
\end{equation}
or the more general homogeneous Robin boundary condition
\begin{equation}
\label{(1.4)}
\ppp_{\nu_A}u + \sigma(x)u = 0 \quad \mbox{on $\ppp\OOO \times (0,T)$},   
\end{equation}
where  $\sigma$ is a sufficiently smooth function on $\ppp\OOO$ that 
satisfies the condition $\sigma(x) \ge 0,\ x\in \ppp\OOO $.

For partial differential equations of the parabolic type 
that correspond to the case $\alpha=1$ in the equation \eqref{(1.1)},   
several 
important qualitative properties of solutions to the corresponding initial-boundary 
value problems are known.  In particular,
we mention a maximum principle and a comparison principle for the solutions to these problems (\cite{PW}, \cite{RR}).

The main purpose of this paper is the comparison principles 
for the linear time-fractional  
diffusion equation \eqref{(1.1)} with the Neumann or the 
Robin boundary conditions.

For the equations of type \eqref{(1.1)} 
with the Dirichlet boundary conditions,  the maximum principles in different formulations  were  derived and 
used in \cite{Bor,Lu1,luchko-1,luchko-2,Lu2,LY1,LY2,LY3,Za}. 
For a maximum principle for the time-fractional transport equations
we refer to \cite{LSY}. In \cite{Kir}, a maximum principle for the more general 
space- and time-space-fractional partial differential equations has been 
derived.

Because any maximum principle involves the Dirichlet boundary values, its 
formulation in the case of the Neumann or Robin boundary 
conditions requires more cares. 
For this kind of the boundary conditions, both positivity of solutions 
and the comparison 
principles can be derived under some suitable restrictions on the problem 
data. One typical result of this sort says
that the solution $u$ to the equation \eqref{(1.1)} with the boundary 
condition \eqref{(1.3a)} or \eqref{(1.4)} 
and an appropriately formulated initial condition is 
non-negative in $\OOO\times (0,T)$ 
if the initial value $a$ and the non-homogeneous term $F$ are non-negative
in $\OOO$ and in $\QQQQ$, respectively.
Such positivity properties and their applications have been intensively discussed  
and used for the partial differential equations of parabolic type ($\alpha=1$ 
in the equation \eqref{(1.1)}), see, e.g., \cite{E}, 
 \cite{Fr}, \cite{Pao2}, or \cite{RR}. 

However, to the best knowledge of the authors, no results of this kind have been published for 
the time-fractional diffusion equations in the case of the Neumann or Robin 
boundary conditions. The main subject of this paper is in derivation of 
a positivity property and the comparison principles for the linear equation 
\eqref{(1.1)} with the boundary condition
\eqref{(1.3a)} or \eqref{(1.4)} and an appropriately formulated initial 
condition. In the subsequent 
paper, these result will be extended to the case of the semilinear 
time-fractional diffusion 
equations.  The arguments employed in these papers 
rely on an operator theoretical approach to the fractional integrals and 
derivatives in the fractional Sobolev spaces that is an extension of the 
theory well-known in the case $\alpha=1$, see, e.g., \cite{He}, \cite{Pa}, \cite{Ta}. We also refer to the recent publications \cite{Al-R} and \cite{Lu} devoted to the comparison principles for solutions to the fractional differential inequalities with the general fractional derivatives and for solutions to the ordinary fractional differential equations, respectively.

The rest of this paper is organized as follows. In Section \ref{sec2},  
some important results regarding the unique existence of solutions 
to the initial-boundary value problems for the linear time-fractional  
diffusion equations are presented. Section \ref{sec3} is devoted to a proof 
of a key lemma that is a basis for the proofs of the comparison principles for the linear and semilinear time-fractional
diffusion equations. 
The lemma asserts that each solution to \eqref{(1.1)} is non-negative 
in $\OOO \times (0,T)$ if $a\ge 0$ and $F \ge 0$, provided that 
$u$ is assumed to satisfy some extra regularity.
In Section \ref{sec4},  we prove a comparison principle that is our
main result for the problem \eqref{(1.1)} for the linear time-fractional 
diffusion equation. Moreover, we establish the order-preserving properties
for other problem data (the zeroth-order coefficient $c$ of the equation and the coefficient $\sigma$ of the Robin condition). 
Finally, a detailed proof of an important auxiliary statement is presented in an Appendix. 

\section{Well-posedness results} 
\label{sec2}

\setcounter{section}{2}
\setcounter{equation}{0}

For $x \in \Omega, \thinspace 0<t<T$,  we define an operator
\begin{equation}
\label{(2.1)}
-Av(x,t) := \sum_{i,j=1}^d \ppp_i(a_{ij}(x)\ppp_j v(x,t)
+ \sum_{j=1}^d b_j(x,t)\ppp_jv(x,t) + c(x,t)v(x,t)     
\end{equation}
and assume that the conditions \eqref{(1.2)} for the coefficients 
$a_{ij}, b_j, c$ are satisfied. 
 
In this section, we deal with the following initial-boundary value problem for 
the linear  time-fractional diffusion equation \eqref{(1.1)} with the time-fractional derivative 
of order $\alpha\in (0,1)$
\begin{equation}
\label{(2.2)}
\left\{ \begin{array}{rl}
& \pppa (u(x,t)-a(x)) + Au(x,t) = F(x,t), 
\quad x \in \Omega, \thinspace 0<t<T, \\
& \NUNU u + \sigma(x)u(x,t) = 0, \quad x\in \ppp\OOO, \, 0<t<T,
\end{array}\right.
\end{equation}
along with the initial condition \eqref{incon} formulated below. 

To appropriately define the  Caputo fractional derivative $\ddda w(t)$, $0<\alpha<1$, 
we start with its definition on the space
$$
\CC := \{ u \in C^1[0,T];\thinspace u(0) = 0\}
$$
that reads as follows:  
$$
\ddda w(t) = \frac{1}{\Gamma(1-\alpha)}\int^t_0
(t-s)^{-\alpha}\frac{dw}{ds}(s) ds,\ w\in 
\CC.
$$
Then we extend this operator from the domain $\mathcal{D}(\ddda)
:= \CC$ to $L^2(0,T)$ taking into account its closability 
(\cite{Yo}).  As have been shown in \cite{KRY}, there exists a unique minimum 
closed 
extension of $\ddda$ with the domain $\mathcal{D}(\ddda) = \CC$.
Moreover, the domain of this 
extension is the closure of $\CC$ in the Sobolev-Slobodeckij space 
$H^{\alpha}(0,T)$.  Let us recall that the norm $\Vert \cdot\Vert_{H^{\alpha}
(0,T)}$ of the Sobolev-Slobodeckij space 
$H^{\alpha}(0,T)$ is defined as follows (\cite{Ad}):
$$
\Vert v\Vert_{H^{\alpha}(0,T)}:=
\left( \Vert v\Vert^2_{L^2(0,T)}
+ \int^T_0\int^T_0 \frac{\vert v(t)-v(s)\vert^2}{\vert t-s\vert^{1+2\alpha}}
dtds \right)^{\hhalf}.
$$
By setting 
$$
H_{\alpha}(0,T):= \ooo{\CC}^{H^{\alpha}(0,T)},
$$
we obtain (\cite{KRY})
$$
H_{\alpha}(0,T) = 
\left\{ \begin{array}{rl}
&H^{\alpha}(0,T), \quad  0<\alpha<\hhalf, \\
&\left\{ v \in H^{\hhalf}(0,T);\, \int^T_0 \frac{\vert v(t)\vert^2}{t}
dt < \infty \right\}, \quad  \alpha=\hhalf, \\
& \{ v \in H^{\alpha}(0,T);\, v(0) = 0\}, \quad \hhalf < \alpha < 1,
\end{array}\right.
$$
and 
$$
\Vert v\Vert_{H_{\alpha}(0,T)} = 
\left\{ \begin{array}{rl}
&\Vert v\Vert_{H^{\alpha}(0,T)}, \quad  \alpha \ne \hhalf, \\
&\left( \Vert v\Vert_{H^{\hhalf}(0,T)}^2
+ \int^T_0 \frac{\vert v(t)\vert^2}{t}dt\right)^{\hhalf}, \quad
\alpha=\hhalf.
\end{array}\right.
$$
In what follows, we also use the Riemann-Liouville fractional
integral operator $J^{\beta}$, $\beta > 0$ defined by 
$$
(J^{\beta}f)(t) := \frac{1}{\Gamma(\beta)}\int^t_0 (t-s)^{\beta-1}f(s) ds,
\quad 0<t<T.
$$
Then, according to \cite{GLY} and \cite{KRY}, 
$$
H_{\alpha}(0,T) = J^{\alpha}L^2(0,T),\quad 0<\alpha<1.
$$
Next we define 
$$
\pppa = (J^{\alpha})^{-1} \quad \mbox{with $\mathcal{D}(\pppa)
= H_{\alpha}(0,T)$}.
$$
As have been shown in \cite{GLY} and \cite{KRY}, there exists a constant $C>0$ depending only on $\alpha$ such that 
$$
C^{-1}\Vert v\Vert_{H_{\alpha}(0,T)} \le \Vert \pppa v\Vert_{L^2(0,T)}
\le C\Vert v\Vert_{H_{\alpha}(0,T)} \quad \mbox{for all } v\in H_{\alpha}(0,T).
$$

Now we can introduce a suitable form of  initial condition for the problem 
\eqref{(2.2)} as follows
\begin{equation}
\label{incon}
u(x, \cdot) - a(x) \in \HH(0,T) \quad \mbox{for almost all } x\in \OOO
\end{equation}
and write down a complete formulation of an initial-boundary value problem 
for the linear  time-fractional diffusion equation \eqref{(1.1)}:  
\begin{equation}
\label{(2.3)}
\left\{ \begin{array}{rl}
& \pppa (u(x,t)-a(x)) + Au(x,t) = F(x,t), 
\quad x \in \Omega, \thinspace 0<t<T, \\
& \NUNU u(x,t) + \sigma(x)u(x,t) = 0, \quad x\in \ppp\OOO, \, 0<t<T,\\
& u(x, \cdot) - a(x) \in \HH(0,T) \quad \mbox{for almost all }x\in \OOO.
\end{array}\right.
\end{equation}

It is worth mentioning that the term $\pppa (u(x,t) - a(x))$  in the first line of \eqref{(2.3)} is 
well-defined due to inclusion formulated in the third line of \eqref{(2.3)}.
In particular, for $\frac{1}{2} < \alpha < 1$, the Sobolev embedding leads 
to the inclusions
$\HH(0,T) \subset H^{\alpha}(0,T)\subset C[0,T]$. 
This means that $u\in \HH(0,T;L^2(\OOO))$ implies  $u \in C([0,T];L^2(\OOO))$  
and thus in this case the initial condition can be formulated as 
$u(\cdot,0) = a$  in $L^2$-sense.  Moreover, for sufficiently smooth functions 
$a$ and $F$, the solution to \eqref{(2.3)} can be proved to satisfy 
the initial condition in a usual sense: $\lim_{t\to 0} 
u(\cdot,t) = a$ in $L^2(\OOO)$ (see Lemma 4 in Section 3).
Consequently, the third line of \eqref{(2.3)} can be  interpreted as 
a generalized initial condition.

In the following theorem, a fundamental result regarding the unique existence 
of the solution to 
the initial-boundary value problem \eqref{(2.3)} is presented.  

\begin{theorem}
\label{t2.1}
For $a\in H^1(\OOO)$ and $F \in L^2(0,T;L^2(\OOO))$, there exists a unique 
solution 
$u(F,a) = u(F,a)(x,t) \in L^2(0,T;H^2(\OOO))$ to the initial-boundary value 
problem \eqref{(2.3)} such that 
$u(F,a)-a \in \HH(0,T;L^2(\OOO))$.

Moreover, there exists a constant $C>0$ such that 
\begin{align*}
& \Vert u(F,a)-a\Vert_{\HH(0,T;L^2(\OOO))} 
+ \Vert u(F,a)\Vert_{L^2(0,T;H^2(\OOO))} \\
\le &C(\Vert a\Vert_{H^1(\OOO)} + \Vert F\Vert_{L^2(0,T;L^2(\OOO))}).
\end{align*}
\end{theorem}

Before starting with a proof of Theorem \ref{t2.1}, we introduce some notations and derive several helpful results needed for the proof. 

For an arbitrary constant $c_0>0$,  we define an elliptic operator $A_0$ 
 as follows:
\begin{equation}
\label{(3.2)}
\left\{ \begin{array}{rl}
& (-A_0v)(x) := \sumij \ppp_i(a_{ij}(x)\ppp_jv(x)) - c_0v(x), \quad
x\in \OOO, \\
& \mathcal{D}(A_0) = \left\{ v\in H^2(\OOO);\,
\NUNU v + \sigma v = 0 \quad \mbox{on } \ppp\OOO \right\}.
\end{array}\right.
\end{equation}

We recall that in the definition \eqref{(3.2)}, $\sigma$ is a smooth function, the inequality $\sigma(x)\ge 0,\ x\in \ppp\OOO$ holds true, 
and the coefficients $a_{ij}$ satisfy the conditions 
\eqref{(1.2)}. 

Henceforth, by $\Vert \cdot\Vert$ and $(\cdot,\cdot)$ we denote the 
standard norm and
the scalar product  in $L^2(\OOO)$, respectively. 
It is well-known that the operator $A_0$ is self-adjoint and its resolvent 
is a compact operator.  Moreover, for a sufficiently large constant
$c_0>0$, by Lemma \ref{lem1} in Section \ref{sec8}, we can verify that $A_0$ is 
positive definite.
Therefore, by choosing the  constant $c_0>0$ large enough,
the spectrum of $A_0$ consists entirely of discrete positive 
eigenvalues $0 < \la_1 \le \la_2 \le \cdots$,
which are numbered according to their multiplicities and  
$\la_n \to \infty$ as $n\to \infty$.
Let $\va_n$ be an eigenvector corresponding to the eigenvalue $\la_n$ such that $A\va_n = \la_n\va_n$ and 
$(\va_n, \va_m) = 0$ if $n \ne m$ and $(\va_n,\va_n) = 1$. 
Then the system $\{ \va_n\}_{n\in \N}$ of the 
eigenvectors forms an
orthonormal basis in $L^2(\OOO)$ and for any $\gamma\ge 0$ we can define
the fractional powers $A_0^{\gamma}$ of the operator $A_0$ 
by the following relation (see, e.g., \cite{Pa}):
$$
A_0^{\gamma}v = \sum_{n=1}^{\infty} \la_n^{\gamma} (v,\va_n)\va_n,
$$
where
$$
v \in \mathcal{D}(A_0^{\gamma})
:= \left\{ v\in L^2(\OOO): \thinspace
\sum_{n=1}^{\infty} \la_n^{2\gamma} (v,\va_n)^2 < \infty\right\}
$$
and 
$$
\Vert A_0^{\gamma}v\Vert = \left( \sum_{n=1}^{\infty}
\la_n^{2\gamma} (v,\va_n)^2 \right)^{\frac{1}{2}}.
$$
We note that $\mathcal{D}(A_0^{\gamma}) \subset H^{2\gamma}(\OOO)$.

Our proof of Theorem \ref{t2.1} is similar to the one presented in \cite{GLY}, \cite{KRY} for the case of the homogeneous Dirichlet 
boundary condition. In particular, we 
employ  the operators $S(t)$ and $K(t)$ defined by (\cite{GLY}, \cite{KRY})
\begin{equation}
\label{(5.1)}
S(t)a = \sum_{n=1}^{\infty} E_{\alpha,1}(-\la_n t^{\alpha})
(a,\va_n)\va_n, \quad a\in L^2(\OOO), \thinspace t>0  
\end{equation}
and
\begin{equation}
\label{(5.2)}
K(t)a = -A_0^{-1}S'(t)a 
= \sum_{n=1}^{\infty} t^{\alpha-1}E_{\alpha,\alpha}(-\la_n t^{\alpha})
(a,\va_n)\va_n, \quad a\in L^2(\OOO), \thinspace t>0.
\end{equation}
In the above formulas, $E_{\alpha,\beta}(z)$ denotes the Mittag-Leffler function defined by 
a convergent series as follows:
$$
E_{\alpha,\beta}(z) = \sum_{k=0}^\infty \frac{z^k}{\Gamma(\alpha\, k + \beta)},
\ \alpha >0,\ \beta \in \mathbb{C},\ z \in \mathbb{C}.
$$
It follows directly from the definitions given above that
$A_0^{\gamma}K(t)a = K(t)A_0^{\gamma}a$
and $A_0^{\gamma}S(t)a = S(t)A_0^{\gamma}a$ for $a \in \mathcal{D}
(A_0^{\gamma})$.
Moreover, the inequality (see, e.g., Theorem 1.6 (p. 35) in \cite{Po})
$$
\max \{ \vert E_{\alpha,1}(-\la_nt^{\alpha})\vert, \, 
\vert E_{\alpha,\alpha}(-\la_nt^{\alpha})\vert \} \le \frac{C}{1+\la_nt^{\alpha}}
\quad \mbox{for all $t>0$}
$$
implicates the estimations (\cite{GLY})
\begin{equation}
\label{(5.3)}
\left\{ \begin{array}{l}
\Vert A_0^{\gamma}S(t)a\Vert \le Ct^{-\alpha\gamma}\Vert a\Vert, \\
\Vert A_0^{\gamma}K(t)a\Vert \le Ct^{\alpha(1-\gamma)-1}
\Vert a\Vert, \quad a \in L^2(\OOO), \thinspace t > 0, \thinspace
0 \le \gamma \le 1.
\end{array}\right.                    
\end{equation}
In order to shorten the notations and to focus on the dependence on the time 
variable $t$, henceforth we sometimes omit the variable  $x$ in the functions 
of two variables $x$ and $t$ and write, say, 
$u(t)$ instead of  $u(\cdot,t)$.

Due to the inequalities \eqref{(5.3)}, the estimations provided in the formulation of Theorem \ref{t2.1} can be derived as in
the case of the fractional powers of generators of the analytic semigroups 
(\cite{He}). To do this, we first formulate and prove the following lemma:

\begin{lemma}
\label{l5.1}
Under the conditions formulated above, the following estimates hold true for 
$F\in L^2(0,T;L^2(\OOO))$ and  $a \in L^2(\OOO)$:

\noindent
(i) 
$$
\left\Vert \int^t_0 A_0K(t-s)F(s) ds \right\Vert_{L^2(0,T;L^2(\OOO))}
\le C\Vert F\Vert_{L^2(0,T;L^2(\OOO))},
$$
\noindent
(ii)
$$
\left\Vert \int^t_0 K(t-s)F(s) ds \right\Vert_{\HH(0,T;L^2(\OOO))}
\le C\Vert F\Vert_{L^2(0,T;L^2(\OOO))},
$$
\noindent
(iii) 
$$
\Vert S(t)a - a\Vert_{\HH(0,T;L^2(\OOO))}
+ \Vert S(t)a\Vert_{L^2(0,T;H^2(\OOO))} \le C\Vert a\Vert.
$$
\end{lemma}

\begin{proof} 
We start with proving the estimate (i). By \eqref{(5.2)}, we have
\begin{align*}
& \left\Vert \int^t_0 A_0 K(t-s)F(s) ds \right\Vert^2\\
=& \left\Vert \sumn \left(\int^t_0 \la_n(t-s)^{\alpha-1}
\MLTWO(-\la_n(t-s)^{\alpha})
(F(s), \va_n) ds\right) \va_n\right\Vert^2\\
=& \sumn \left\vert \int^t_0 \la_n(t-s)^{\alpha-1}\MLTWO(-\la_n(t-s)^{\alpha})
(F(s), \va_n) ds \right\vert^2.
\end{align*}
Therefore, using the Parseval equality and the Young inequality for the 
convolution, we obtain
\begin{align*} 
& \left\Vert \int^t_0 A_0K(t-s)F(s) ds \right\Vert^2_{L^2(0,T;L^2(\OOO))}\\
= & \sumn \int^T_0 \vert (\la_ns^{\alpha-1}\MLTWO(-\la_ns^{\alpha}) \, *\,
(F(s), \va_n) \vert^2 ds\\
= & \sumn \Vert \la_ns^{\alpha-1}\MLTWO(-\la_ns^{\alpha}) \, * \,
(F(s), \va_n) \Vert^2_{L^2(0,T)}\\
\le & \sumn \left( \la_n\int^t_0 \vert t^{\alpha-1}\MLTWO(-\la_nt^{\alpha})
\vert dt \right)^2 \Vert (F(t),\va_n)\Vert^2_{L^2(0,T)}.
\end{align*} 
Then we employ the representation
\begin{equation}\label{(2.9a)}
\frac{d}{dt}\MLONE(-\la_nt^{\alpha}) 
= -\la_nt^{\alpha-1}\MLTWO(-\la_nt^{\alpha}),
\end{equation}
and the complete monotonicity of the Mittag-Leffler function (\cite{GKMR})
$$
\MLONE(-\la_nt^{\alpha}) > 0, \quad
\frac{d}{dt}\MLONE(-\la_nt^{\alpha}) \le 0, \quad t\ge 0, \quad 0<\alpha\le 1
$$
to get the inequality
\begin{equation}
\label{(5.4)}
\int^T_0 \vert \la_nt^{\alpha-1}\MLTWO(-\la_nt^{\alpha})\vert dt 
= -\int^T_0 \frac{d}{dt}\MLONE(-\la_nt^{\alpha})dt
\end{equation}
$$
= 1 - \MLONE(-\la_nT^{\alpha}) \le 1 \quad \mbox{for all $n\in \N$}.
$$
Hence,
\begin{align*}
& \left\Vert \int^t_0 A_0K(t-s)F(s) ds \right\Vert^2_{L^2(0,T;L^2(\OOO))}
\le \sumn \Vert (F(t), \va_n)\Vert^2_{L^2(0,T)}\\
=& \int^T_0 \sumn \vert (F(t), \va_n) \vert^2 dt
= \int^T_0 \Vert F(\cdot,t)\Vert^2 dt 
= \Vert F\Vert_{L^2(0,T;L^2(\OOO))}^2.
\end{align*}

Now we proceed with proving the estimate (ii). For $0<t<T, \, n\in \N$
and $f\in L^2(0,T)$,
we set 
$$
(L_nf)(t) := \int^t_0 (t-s)^{\alpha-1}\MLTWO(-\la_n(t-s)^{\alpha}) 
f(s) ds.
$$
Then 
$$
\int^t_0 K(t-s)F(s) ds = \sumn (L_nf)(t)\va_n
$$
in $L^2(\OOO)$ for any fixed $t \in [0,T]$.

First we prove that 
\begin{equation}
\label{(5.5)}
\left\{ \begin{array}{rl}
& L_nf \in \HH(0,T), \\
& \pppa(L_nf)(t) = -\la_nL_nf(t) + f(t), \quad 0<t<T, \\
& \Vert L_nf\Vert_{\HH(0,T)} \le C\Vert f\Vert_{L^2(0,T)},
\quad n\in \N \quad \mbox{for each } f\in L^2(0,T).            
\end{array}\right.
\end{equation}
In order to prove this, we apply the Riemann-Liouville fractional 
integral  $J^{\alpha}$ to $L_nf$ and get the representation
\begin{align*}
& J^{\alpha}(L_nf)(t) 
= \frac{1}{\Gamma(\alpha)}\int^t_0 (t-s)^{\alpha-1}(L_nf)(s) ds\\
=& \frac{1}{\Gamma(\alpha)} \int^t_0 (t-s)^{\alpha-1}
\left( \int^s_0  (s-\xi)^{\alpha-1}\MLTWO(-\la_n (s-\xi)^{\alpha})f(\xi) d\xi
\right) ds\\
=& \frac{1}{\Gamma(\alpha)}\int^t_0 f(\xi) \left( 
\int^t_{\xi} (t-s)^{\alpha-1}(s-\xi)^{\alpha-1}\MLTWO(-\la_n (s-\xi)^{\alpha})
ds \right) d\xi.
\end{align*}
By direct calculations, using \eqref{(2.9a)}, we obtain the formula
\begin{align*}
& \frac{1}{\Gamma(\alpha)}\int^t_{\xi} (t-s)^{\alpha-1}(s-\xi)^{\alpha-1}
\MLTWO(-\la_n (s-\xi)^{\alpha}) ds\\
=&  -\frac{1}{\la_n}(t-\xi)^{\alpha-1}\left(\MLTWO(-\la_n t^{\alpha}) 
- \frac{1}{\Gamma(\alpha)}\right).
\end{align*}

Therefore, we have the relation 
\begin{align*}
& J^{\alpha}(L_nf)(t) = -\frac{1}{\la_n}(L_nf)(t)
+ \frac{1}{\la_n}\int^t_0 (t-\xi)^{\alpha-1}\frac{1}{\Gamma(\alpha)} 
f(\xi) d\xi\\
=& -\frac{1}{\la_n}(L_nf)(t) + \frac{1}{\la_n}(J^{\alpha}f)(t),
\quad n\in \N,
\end{align*}
that is,
$$
(L_nf)(t) = -\la_n J^{\alpha}(L_nf)(t) + (J^{\alpha}f)(t), \quad 0<t<T.
$$
Hence, $L_nf \in \HH(0,T) = J^{\alpha}L^2(0,T)$.  By definition, 
$\pppa = (J^{\alpha})^{-1}$ (\cite{KRY}) and thus the last  formula can be rewritten in the form
$$
\pppa (L_nf) = -\la_n L_nf + f \quad \mbox{in }(0,T).
$$

Using the inequality \eqref{(5.4)}, we obtain
$$
\la_n\Vert L_nf\Vert_{L^2(0,T)} \le \la_n\Vert s^{\alpha-1}
\MLTWO(-\la_ns^{\alpha})\Vert_{L^1(0,T)}\Vert f\Vert_{L^2(0,T)}
\le \Vert f\Vert_{L^2(0,T)}.
$$
Therefore,
\begin{align*}
& \Vert L_nf\Vert_{\HH(0,T)} \le C\Vert \pppa(L_nf)\Vert_{L^2(0,T)}
\le C(\Vert -\la_nL_nf\Vert_{L^2(0,T)} + \Vert f\Vert_{L^2(0,T)})\\
\le& C\Vert f\Vert_{L^2(0,T)}, \quad n\in \N, \quad f\in L^2(0,T).
\end{align*}
Thus, the estimate \eqref{(5.5)} is proved.

Now we set $f_n(s) := (F(s), \, \va_n)$ for $0<s<T$ and $n\in \N$.
Since 
$$
\pppa \int^t_0 K(t-s)F(s) ds = \sumn \pppa(L_nf_n)(t)\va_n,
$$
we obtain
$$
\left\Vert \pppa \int^t_0 K(t-s)F(s) ds\right\Vert^2_{L^2(\OOO)}
= \sumn \vert \pppa(L_nf_n)(t)\vert^2.
$$ 
By applying \eqref{(5.5)}, we get the following chain of inequalities and equations: 
\begin{align*}
\left\Vert \pppa \int^t_0 K(t-s)F(s) ds\right\Vert^2
_{\HH(0,T;L^2(\OOO))}
\le C\left\Vert \pppa\int^t_0 K(t-s)F(s) ds\right\Vert^2
_{L^2(0,T;L^2(\OOO))}\\
= C\sumn \Vert \pppa(L_nf_n)\Vert^2_{L^2(0,T)}
\le C\sumn \Vert L_nf_n\Vert^2_{\HH(0,T)} \le C\sumn \Vert f_n\Vert^2_{L^2(0,T)}\\
=  C\int^T_0 \sumn \vert (F(s),\va_n) \vert^2 ds
= C\int^T_0 \Vert F(s)\Vert_{L^2(\OOO)}^2 ds 
= C\Vert F\Vert^2_{L^2(0,T;L^2(\OOO))}.
\end{align*}
Thus, the proof of the estimate (ii) is completed.

The estimate (iii) from Lemma \ref{l5.1} follows from the standard estimates
of the operator $S(t)$. It can be derived by the same arguments as those that 
were employed in 
Section 6 of Chapter 4 in \cite{KRY} for the case of the 
homogeneous Dirichlet boundary condition and we omit here  the technical details. 
\end{proof}

Now we proceed to the proof of Theorem \ref{t2.1}.
 
 \begin{proof}
 
In the first line of the problem \eqref{(2.3)}, we regard the expressions 
$\sum_{j=1}^d b_j(x,t)\ppp_ju$ and 
$c(x,t)u$ as some non-homogeneous terms. Then this problem can be rewritten in terms  of 
the operator $A_0$ as follows
\begin{equation}
\label{(5.6)}
\left\{ \begin{array}{rl}
& \pppa (u-a) + A_0u(x,t) = F(x,t)\\
+& \sum_{j=1}^d b_j(x,t)\ppp_ju + (c_0+c(x,t))u, \quad 
x\in \OOO,\, 0<t<T,\\
& \NUNU u + \sigma(x) u = 0 \quad \mbox{on }\ppp\OOO \times (0,T),\\
& u(x,\cdot) - a(x) \in \HH(0,T) \quad \mbox{for almost all }x\in \OOO.
\end{array}\right.
\end{equation}
In its turn, the first line of \eqref{(5.6)} can be represented in the form  (\cite{GLY}, 
\cite{KRY})
$$
u(t) = S(t)a + \int^t_0 K(t-s)F(s) ds 
$$
\begin{equation}
\label{(5.7)}
+ \int^t_0 K(t-s) \left(\sum_{j=1}^d b_j(s)\ppp_ju(s) 
+ (c_0+c(s))u(s) \right) ds, \quad 0<t<T.        
\end{equation}
Moreover, it is known that if 
$u\in L^2(0,T;H^2(\OOO))$ satisfies the initial condition 
$u-a \in \HH(0,T;L^2(\OOO))$ and 
the equation \eqref{(5.7)}, then $u$ is a solution to the problem 
\eqref{(5.6)}.
With the notations
\begin{equation}
\label{(5.8)}
\left\{ \begin{array}{rl}
& G(t):=  \int^t_0 K(t-s)F(s) ds + S(t)a, \\
& Qv(t) = Q(t)v(t) := \sum_{j=1}^d b_j(\cdot,t)\ppp_jv(t) 
+ (c_0+c(\cdot,t))v(t), \\
& Rv(t):= \int^t_0 K(t-s) \left(\sum_{j=1}^d b_j(\cdot,s)\ppp_jv(s) 
+ (c_0+c(\cdot,s))v(s) \right) ds,\\
&\qquad \qquad \qquad \mbox{for $0<t<T$}, 
\end{array}\right.
\end{equation}
the equation \eqref{(5.7)} can be represented in form of a fixed point equation
$u = Ru + G$ on the space $L^2(0,T;H^2(\OOO))$.

Lemma \ref{l5.1} yields the inclusion $G \in L^2(0,T;H^2(\OOO))$.
Moreover, since $\Vert A_0^{\hhalf}a\Vert \le C\Vert a\Vert_{H^1(\OOO)}$
and $\DDD(A_0^{\hhalf}) = H^1(\OOO)$ (see, e.g., \cite{Fu}),
the estimate \eqref{(5.3)} implies
$$
\Vert S(t)a\Vert_{H^2(\OOO)} \le C\Vert A_0S(t)a\Vert
= C\Vert A_0^{\hhalf}S(t)A_0^{\hhalf}a\Vert
\le Ct^{-\hhalf\alpha}\Vert a\Vert_{H^1(\OOO)}
$$
and thus
$$
\Vert S(t)a\Vert^2_{L^2(0,T;H^2(\OOO))} 
\le C\left(\int^T_0 t^{-\alpha}dt \right)\Vert a\Vert^2_{H^1(\OOO)}
\le \frac{CT^{1-\alpha}}{1-\alpha}\Vert a\Vert^2_{H^1(\OOO)}.
$$
Consequently, the inclusion $S(t)a \in L^2(0,T;H^2(\OOO))$ holds valid.

For $0<t<T$, we next estimate $\Vert Rv(\cdot,t)\Vert_{H^2(\OOO)}$ for 
$v(\cdot,t) \in \mathcal{D}(A_0)$ as follows:
\begin{align*}
& \Vert Rv(\cdot,t)\Vert_{H^2(\OOO)}
\le C\Vert A_0Rv(\cdot,t)\Vert_{L^2(\OOO)}\\
\le & \int^t_0 \left\Vert A_0^{\hhalf}K(t-s)A_0^{\hhalf}
\left(\sum_{j=1}^d b_j(s)\ppp_jv(s) + (c_0+c(s))v(s) \right)
\right\Vert ds\\
\le & C\int^t_0 \Vert A_0^{\hhalf}K(t-s)\Vert 
\left\Vert A_0^{\hhalf}
\left(\sum_{j=1}^d b_j(s)\ppp_jv(s) + (c_0+c(s))v(s) \right)
\right\Vert ds\\
\le& C\int^t_0 (t-s)^{\hhalf\alpha -1}\Vert v(s)\Vert_{H^2(\OOO)}ds
= C\left( \Gamma\left(\hhalf\alpha\right)J^{\hhalf\alpha}\Vert v\Vert
_{H^2(\OOO)}\right)(t).
\end{align*}
For derivation of this estimate, we employed the inequalities
$$
\Vert A_0^{\hhalf}b_j(s)\ppp_jv(t)\Vert 
\le C\Vert b_j(s)\ppp_jv(s)\Vert_{H^1(\OOO)}
\le C\Vert v(s)\Vert_{H^2(\OOO)}
$$
and $\Vert (c(s)+c_0)v(s)\Vert_{H^1(\OOO)} \le C\Vert v(s)\Vert_{H^2(\OOO)}$ that are valid 
because of the inclusions $b_j \in C^1(\ooo{\OOO}\times [0,T])$) and 
$c+c_0\in C([0,T];C^1(\ooo{\OOO}))$.

Since $(J^{\hhalf\alpha}w_1)(t) \ge (J^{\hhalf\alpha}w_2)(t)$ 
if $w_1(t) \ge w_2(t)$ for $0\le t\le T$, and
$J^{\hhalf\alpha}J^{\hhalf\alpha}w = J^{\alpha}w$ for 
$w_1, w_2, w \in L^2(0,T)$, we have
\begin{align*}
&\Vert R^2v(t)\Vert_{H^2(\OOO)} = \Vert R(Rv)(t)\Vert_{H^2(\OOO)}\\
\le & C\left( \Gamma\left(\hhalf\alpha\right)J^{\hhalf\alpha}
\left(C\Gamma\left(\hhalf\alpha\right)J^{\hhalf\alpha}
\Vert v\Vert_{H^2(\OOO)}\right) \right)(t)\\
= & \left( C\Gamma\left(\hhalf\alpha\right)\right)^2
(J^{\alpha}\Vert v\Vert_{H^2(\OOO)})(t).
\end{align*}
Repeating this argumentation $m$-times, we obtain
\begin{align*}
& \Vert R^mv(t)\Vert_{H^2(\OOO)}
\le \left( C\Gamma\left(\hhalf\alpha\right)\right)^m
\left( J^{\hhalf\alpha m}\Vert v\Vert_{H^2(\OOO)}\right)(t)\\
\le & \frac{\left( C\Gamma\left(\hhalf\alpha\right)\right)^m}
{\Gamma\left( \hhalf\alpha m\right)}
\int^t_0 (t-s)^{\frac{m}{2}\alpha -1} \Vert v(\xi)\Vert_{H^2(\OOO)}ds, \quad 
0<t<T.
\end{align*}
Applying the Young inequality to the integral at the right-hand side of the last estimate, 
we arrive to the inequality
\begin{align*}
& \Vert R^mv(t)\Vert_{L^2(0,T; H^2(\OOO))}^2
\le \left(  \frac{\left( C\Gamma\left(\hhalf\alpha\right) \right)^m}
{\Gamma\left( \hhalf\alpha m\right)}\right)^2
\Vert t^{\frac{\alpha m}{2}-1}\Vert_{L^1(0,T)}^2 
\Vert v\Vert_{L^2(0,T;H^2(\OOO))}^2\\
=& \frac{\left( CT^{\frac{\alpha}{2}}
\Gamma\left(\hhalf\alpha \right)\right)^{2m}}
{\Gamma\left( \hhalf\alpha m +1\right)^2}
\Vert v\Vert_{L^2(0,T;H^2(\OOO))}^2.
\end{align*}
Employing the known asymptotic behavior of the gamma function, we obtain the 
relation
$$
\lim_{m\to\infty} \frac{\left( CT^{\frac{\alpha}{2}}
\Gamma\left(\hhalf\alpha \right)\right)^m}
{\Gamma\left( \hhalf\alpha m +1\right)} = 0
$$
that means that for sufficiently large $m\in \N$, the mapping
$$
R^m: L^2(0,T;H^2(\OOO))\, \longrightarrow \, L^2(0,T;H^2(\OOO))
$$ 
is a contraction.  Hence, by the Banach fixed point theorem, the equation 
\eqref{(5.7)} possesses a unique fixed point.
Therefore, by the first equation in \eqref{(2.3)}, we obtain the inclusion
$\pppa (u-a) \in L^2(0,T;L^2(\OOO))$.
Since $\Vert \eta\Vert_{\HH(0,T)} \sim \Vert \pppa \eta\Vert_{L^2(0,T)}$
for $\eta \in \HH(0,T)$ (\cite{KRY}), we finally obtain the estimate
$$
\Vert u-a\Vert_{\HH(0,T;L^2(\OOO))} + \Vert u\Vert_{L^2(0,T;H^2(\OOO))}
\le C(\Vert a\Vert_{H^1(\OOO)} + \Vert F\Vert_{L^2(0,T;L^2(\OOO))}).
$$
The proof of Theorem \ref{t2.1} is completed.
\end{proof}

\section{Key lemma} 
\label{sec3}

\setcounter{section}{3}
\setcounter{equation}{0}

For derivation of the comparison principles for solutions to the initial-boundary value problems for the linear and 
semilinear time-fractional diffusion equations, we need some auxiliary results 
that are formulated and proved in this section.


In addition to the operator $-A_0$ defined by \eqref{(3.2)}, 
we define an elliptic operator $-A_1$ with a positive zeroth-order
coefficient:
\begin{equation}\label{(3.1a)}
(-A_1(t)v)(x):= (-A_1v)(x) 
\end{equation}
$$
:= \sumij \ppp_i(a_{ij}(x)\ppp_jv(x)) 
+ \sum_{j=1}^d b_j(x,t)\ppp_jv - b_0(x,t)v,
$$
where $b_0 \in C^1([0,T];C^1(\ooo{\OOO})) \cap 
C([0,T];C^2(\ooo{\OOO}))$, $b_0(x,t) > 0,\ (x,t)\in \ooo{\OOO}\times 
[0,T]$, and 
$\min_{(x,t)\in \ooo{\OOO}\times 
[0,T]} b_0(x,t)$ is sufficiently large.

We also recall that for $y\in W^{1,1}(0,T)$, the pointwise Caputo derivative $\ddda$ is defined by 
\begin{equation}
\label{(4.1)}
\ddda y(t) = \frac{1}{\Gamma(1-\alpha)}
\int^t_0 (t-s)^{-\alpha}\frac{dy}{ds}(s) ds.          
\end{equation}


In what follows, we employ an extremum principle for the Caputo fractional 
derivative formulated below.

\begin{lemma}[\cite{Lu1}]
\label{l4.1}
Let the inclusions $y\in C[0,T]$ and $t^{1-\alpha}y' \in C[0,T]$ hold true.  

If the function $y=y(t)$ attains its minimum over the interval 
$[0,T]$ at the point $t_0 \in (0, \,T]$, then 
$$
\ddda y(t_0) \le 0.
$$
\end{lemma}

In Lemma \ref{l4.1}, the assumption $t_0>0$ is essential. This lemma  was formulated and proved in \cite{Lu1} under a 
weaker 
regularity condition posed on the function $y$, but for our arguments  
we can assume that $y\in C[0,T]$ and 
$t^{1-\alpha}y' \in C[0,T]$.

Employing Lemma \ref{l4.1}, we now formulate and prove our key lemma that is a basis for 
further derivations in this paper.

\begin{lemma}[Positivity of a smooth solution]
\label{l4.2}
For $F\in L^2(0,T;L^2(\OOO))$ and $a\in H^1(\OOO)$, let $F(x,t) \ge 0,\ (x,t)\in \OOO\times (0,T)$, $a(x)\ge 0,\ x\in \OOO$, and 
$\min_{(x,t)\in \ooo{\OOO}\times 
[0,T]} b_0(x,t)$ be a sufficiently large positive constant.
Furthermore, we assume that there exists a solution 
$u\in C([0,T];C^2(\ooo{\OOO}))$
to the initial-boundary value problem
\begin{equation}
\label{(4.2)}
\left\{ \begin{array}{rl}
& \pppa (u-a) + A_1u = F(x,t), \quad x\in \OOO,\, 0<t<T, \\
& \NUNU u + \sigma(x)u = 0 \quad \mbox{on } \ppp\OOO\times (0,T),\\
& u(x,\cdot) - a(x) \in \HH(0,T) \quad \mbox{for almost all } x\in \OOO
\end{array} \right. 
\end{equation}
and $u$ satisfies the condition $t^{1-\alpha}\ppp_tu \in 
C([0,T];C(\ooo{\OOO}))$. 

Then the solution $u$ is non-negative: 
$$
u (x,t)\ge 0,\ (x,t)\in \OOO \times (0,T).   
$$
\end{lemma}

For the partial differential equations of parabolic type
with the Robin boundary condition  ($\alpha=1$ in \eqref{(4.2)}), a similar positivity property is well-known.
However, it is worth mentioning  that the regularity of the solution to the problem \eqref{(4.2)} 
at the point $t=0$ is a more delicate question  compared to the one in the case 
$\alpha=1$. In particular, we cannot expect the inclusion $u(x,\cdot) \in 
C^1[0,T]$.  This can be illustrated by a simple example of the equation 
$\pppa y(t) = y(t)$ with $y(t)-1 \in \HH(0,T)$ whose unique solution 
$y(t) = \MLONE(t^{\alpha})$ does not belong to the space $C^1[0,T]$. 

\begin{proof}
First we introduce an auxiliary function $\psi \in C^1([0,T];C^2(\ooo{\OOO}))$ 
that satisfies the conditions
\begin{equation}
\label{(4.3)}
\left\{ \begin{array}{rl}
& A_1\psi(x,t) = 1, \quad (x,t) \in \OOO\times [0,T], \\
& \NUNU \psi + \sigma \psi = 1 \quad \mbox{on } \ppp\OOO\times [0,T].
\end{array}\right.
\end{equation}
Proving existence of such function $\psi$ is non-trivial. In this section, we focus on the proof of the lemma and then come back to the problem  \eqref{(4.3)} in Appendix.

Now, choosing $M>0$ sufficiently large and $\ep>0$ sufficiently small,
we set 
\begin{equation}
\label{(w_u)}
w(x,t) := u(x,t) + \ep(M + \psi(x,t) + t^{\alpha}), \quad x\in \OOO,\,
0<t<T.
\end{equation}

For a fixed $x\in \OOO$, by the assumption on the regularity 
of $u$, we have the inclusion
\begin{equation}\label{(3.6)}
t^{1-\alpha}\ppp_tu(x,\cdot) \in C[0,T].
\end{equation}
Then, $\ppp_tu(x,\cdot) \in L^1(0,T)$, that is, $u(x,\cdot)
\in W^{1,1}(0,T)$.  Moreover,
\begin{equation}\label{(3.7)}
u(x,0) - a(x) = 0, \quad x\in \OOO.
\end{equation}
           
On the other hand, for $w\in H_{\alpha}(0,T) \cap W^{1,1}(0,T)$ 
and $w(0) = 0$, the equality
$$
\pppa w = \ddda w = \ddda (w+c)
$$
holds true with any constant $c$ (see, e.g., Theorem 2.4 of Chapter 2 in \cite{KRY}).

Since $u(x,\cdot) - a \in H_{\alpha}(0,T)$ and $u(x,\cdot) \in W^{1,1}(0,T)$, by \eqref{(3.7)}, the relations
$\pppa (u-a) = \ddda (u-a) = \ddda u$ hold true for almost all $x\in \OOO$.

Furthermore, since $\ep(M+\psi(\cdot,t)+t^{\alpha}) \in W^{1,1}(0,T)$, 
we obtain
\begin{align*}
& \ddda w = \ddda (u + \ep(M+\psi(x,t)+t^{\alpha})))
= \ddda u + \ep\ddda (M+\psi(x,t)+t^{\alpha})\\
=& \pppa (u-a) + \ep(\ddda (\psi + t^{\alpha}))
= \pppa (u-a) + \ep(\ddda\psi + \Gamma(\alpha+1))
\end{align*}
and
\begin{align*}
& A_1w = A_1u + \ep A_1\psi + \ep A_1t^{\alpha} + \ep A_1M\\
=& A_1u + \ep + \ep b_0(x,t)t^{\alpha} + b_0(x,t)\ep M.
\end{align*}
 
Now we choose a constant $M>0$ such that $M + \psi(x,t) \ge 0$ and
$\ddda \psi(x,t) + b_0(x,t)M > 0$ for $(x,t) \in \ooo{\OOO} \times [0,T]$, so 
that 
\begin{equation}
\label{(4.4)}
\ddda w + A_1w 
\end{equation}
$$
= F + \ep(\Gamma(\alpha+1) + \ddda \psi + 1 
+ b_0(x,t)t^{\alpha} + b_0(x,t) M) > 0 \quad \mbox{in } \OOO\times (0,T).
$$
Moreover, because of the relation $\NUNU w = \NUNU u + \ep \NUNU \psi$, 
we obtain the following estimate:
$$
\NUNU w + \sigma w = \NUNU u + \sigma u + \ep + \sigma \ep t^{\alpha}
+ \ep\sigma M
$$
\begin{equation}
\label{(4.5)}
\ge \ep + \sigma\ep t^{\alpha} + \ep\sigma M \ge \ep \quad 
\mbox{on } \ppp\OOO\times (0,T).
\end{equation}
Evaluation of the representation \eqref{(w_u)} at the point $t=0$ immediately leads 
to the formula
$$
w(x,0) = u(x,0) + \ep(\psi(x,0) + M), \quad x\in \OOO.
$$

Let us assume that the inequality 
$$
\min_{(x,t)\in \ooo{\OOO}\times [0,T]} w(x,t) \ge 0
$$
does not hold valid, that is, there exists a point 
$(x_0,t_0) \in \ooo{\OOO}\times [0,T]$ such that 
\begin{equation}
\label{(4.7)}
w(x_0,t_0):= \min_{(x,t)\in \ooo{\OOO}\times [0,T]} w(x,t) < 0.
\end{equation}
Since $M>0$ is sufficiently large and $u(x,0)$ is non-negative, we obtain the inequality 
$$
w(x,0) = u(x,0) + \ep(\psi(x,0) + M) \ge u(x,0) \ge 0, \quad 
x\in \ooo{\OOO},
$$
and thus $t_0$ cannot be zero.

Next, we show that $x_0 \not\in \ppp\OOO$.  Indeed, let us assume that 
$x_0 \in \ppp\OOO$.  Then the estimate \eqref{(4.5)} yields that 
$\NUNU w(x_0,t_0) + \sigma(x_0)w(x_0,t_0) \ge \ep$.
By \eqref{(4.7)} and $\sigma(x_0)\ge 0$, we obtain 
$$
\NUNU w(x_0,t_0) \ge -\sigma(x_0)w(x_0,t_0) + \ep 
\ge \ep > 0,
$$
which implies
$$
\ppp_{\nu_A}w(x_0,t_0) = 
\sumij a_{ij}(x_0)\nu_j(x_0)\ppp_iw(x_0,t_0)
= \nabla w(x_0,t_0) \cdot \mathcal{A}(x_0)\nu(x_0)
$$
\begin{equation}
\label{(4.8)}
= \sum_{i=1}^d (\ppp_iw)(x_0,t_0)[\mathcal{A}(x_0)\nu(x_0)]_i > 0.
\end{equation}
Here $\mathcal{A}(x) = (a_{ij}(x))_{1\le i,j\le d}$ and 
$[b]_i$ means the $i$-th element of a vector $b$.

For sufficiently small $\ep_0>0$ and $x_0\in \ppp\OOO$, we now verify the inclusion
\begin{equation}
\label{(4.9)}
x_0 - \ep_0\mathcal{A}(x_0)\nu(x_0) \in \OOO.             
\end{equation}
Indeed, since the matrix 
$\mathcal{A}(x_0)$ is positive-definite, the inequality  
$$
(\nu(x_0)\, \cdot \, -\ep_0\mathcal{A}(x_0)\nu(x_0))
= -\ep_0(\mathcal{A}(x_0)\nu(x_0)\, \cdot \, \nu(x_0)) < 0
$$
holds true. In other words, the inequality
$$
\angle (\nu(x_0), \, (x_0 - \ep_0\mathcal{A}(x_0)\nu(x_0)) - x_0)
> \frac{\pi}{2}
$$
is satisfied. 
Because the boundary $\ppp\OOO$ is smooth, the domain $\OOO$ is locally 
located on 
one side of $\ppp\OOO$. In a small neighborhood  of the point 
$x_0\in \ppp\OOO$, the boundary $\ppp\OOO$ can 
be described in the local coordinates composed of 
its tangential component in $\R^{d-1}$
and the normal component along $\nu(x_0)$.  
Consequently, if $y \in \R^d$ satisfies the inequality
$\angle (\nu(x_0), y-x_0) > \frac{\pi}{2}$, then $y\in \OOO$.
Therefore, for a sufficiently small $\ep_0>0$, 
the point $x_0-\ep_0\mathcal{A}(x_0)\nu(x_0)$ is located in $\OOO$ and 
we have proved the inclusion \eqref{(4.9)}. 

Moreover, for sufficiently small $\ep_0>0$, we can prove that
\begin{equation}\label{(3.11a)}
w(x_0 - \ep_0\mathcal{A}(x_0)\nu(x_0),\,t_0) < w(x_0,t_0).
\end{equation}
Indeed, the inequality \eqref{(4.7)} yields
$$
\sum_{i=1}^d (\ppp_iw)(x_0-\eta\mathcal{A}(x_0)\nu(x_0),\,t_0)
[\mathcal{A}(x_0)\nu(x_0)]_i > 0 \quad\mbox{if }\vert \eta\vert < \ep_0.
$$
Then, by the mean value theorem, we obtain the inequality 
\begin{align*}
&w(x_0 - \xi\mathcal{A}(x_0)\nu(x_0),\,t_0) - w(x_0,t_0)\\
= & \xi\sum_{i=1}^d \ppp_iw(x_0 - \theta\mathcal{A}(x_0)\nu(x_0),\,t_0) 
(-[\mathcal{A}(x_0)\nu(x_0)]_i) < 0,
\end{align*}
where $\theta$ is a number between $0$ and $\xi\in (0,\ep_0)$.
Thus, the inequality \eqref{(3.11a)} is verified. 

By combining \eqref{(3.11a)} with \eqref{(4.9)}, we 
conclude that there exists a point $\www{x_0} \in \OOO$ such that the inequality 
$w(\www{x_0},t_0) < w(x_0,t_0)$ holds true, which contradicts the assumption 
\eqref{(4.7)}. Thus, 
we have proved that $x_0 \not\in \ppp\OOO$.

According to \eqref{(4.7)}, the function $w$ attains its minimum 
at the point $(x_0,t_0)$. Because $0 < t_0 \le T$, 
Lemma \ref{l4.1} yields the inequality 
\begin{equation}
\label{(4.10)}
\ddda w(x_0,t_0) \le 0.                 
\end{equation}
Since $x_0 \in \OOO$, the necessary condition for an extremum point leads to the equality
\begin{equation}
\label{(4.11)}
\nabla w(x_0,t_0) = 0.                      
\end{equation}
Moreover, because the function $w$ attains its minimum at the point $x_0 \in \OOO$, 
in view of the sign of the Hessian, the inequality
\begin{equation}
\label{(4.12)}
\sumij a_{ij}(x_0)\ppp_i\ppp_j w(x_0,t_0) \ge 0   
\end{equation}
holds true (see, e.g., the proof of Lemma 1 in Section 1 of Chapter 2 in 
\cite{Fr}).

The inequalities $b(x_0,t_0)>0$, $w(x_0,t_0) < 0$, and
\eqref{(4.10)}-\eqref{(4.12)} lead to the estimate  
\begin{align*}
& \ddda w(x_0,t_0) + A_1w(x_0,t_0)\\
= & \ddda w (x_0,t_0) - \sumij a_{ij}(x_0)\ppp_i\ppp_jw(x_0,t_0)
- \sum_{i=1}^d (\ppp_ia_{ij})(x_0)\ppp_jw(x_0,t_0)\\
-& \sum_{i=1}^d b_i(x_0,t_0)\ppp_iw(x_0,t_0) + b(x_0,t_0)w(x_0,t_0) < 0,
\end{align*}
which contradicts the inequality \eqref{(4.4)}.

Thus, we have proved that 
$$
u(x,t) + \ep(M+\psi(x,t)+t^{\alpha}) = w(x,t) \ge 0, \quad 
(x,t) \in \OOO\times (0,T).
$$
Since $\ep>0$ is arbitrary, we let $\ep \downarrow 0$ to obtain the inequality 
$u(x,t) \ge 0$ for $(x,t) \in \OOO\times (0,T)$ and 
 the proof of Lemma \ref{l4.2} is completed. 
\end{proof}

Let us finally mention that the positivity of the function $b_0$ 
from the definition of the operator $-A_1$ 
is an essential condition for validity of our proof of Lemma \ref{l4.2}. 
However, in the next 
section, we remove this condition while deriving the comparison principles for 
the solutions to the initial-boundary value problem \eqref{(2.3)}.
%
%
%
\section{Comparison principles} 
\label{sec4}

\setcounter{section}{4}
\setcounter{equation}{0}

According to the results formulated in Theorem \ref{t2.1}, in this section, we 
consider the solutions to the initial-boundary value problem \eqref{(2.3)} 
that belong to the following space of functions:
\begin{equation}\label{(4.1a)}
\mathcal{Y}_\alpha := \{ u; \, u-a\in \HH(0,T;L^2(\OOO)), \, u\in L^2(0,T;H^2(\OOO))\}.
\end{equation}
In what follows, by $u(F,a)$ we denote the solution to the problem \eqref{(2.3)} with the initial data $a$ and the source function $F$.

Our first result concerning the comparison
principles for the solutions to the initial-boundary value problems for the linear time-fractional diffusion equation is presented in the next theorem. 

\begin{theorem}
\label{t2.2}
Let the functions $a \in H^1(\OOO)$ and $F \in L^2(\OOO\times (0,T))$ satisfy the inequalities $F(x,t) \ge 0,\ (x,t)\in \OOO\times (0,T)$ and
$a(x) \ge 0,\ x\in \OOO$, respectively. 

Then the solution $u(F,a) \in \mathcal{Y}_\alpha$ 
to the initial-boundary value problem \eqref{(2.3)}  is 
non-negative, e.g., the inequality 
$$
u(F,a)(x,t) \ge 0,\ (x,t)\in \OOO\times (0,T)
$$
holds true. 
\end{theorem}

Let us emphasize that the non-negativity of the solution $u$ to the problem \eqref{(2.3)}  holds true for the  space $ \mathcal{Y}_\alpha$ 
and thus $u$ does not necessarily satisfy the inclusions 
$u \in C([0,T];C^2(\ooo{\OOO}))$ and $t^{1-\alpha}\ppp_tu \in 
C([0,T]; C(\ooo{\OOO}))$. Therefore, Theorem \ref{t2.2} is widely applicable. Before presenting its proof, let us discuss one of its corollaries in form of a comparison property:

\begin{corollary}
\label{c2.1}
Let $a_1, a_2  \in H^1(\OOO)$ and $F_1, F_2 \in L^2(\OOO\times (0,T))$
satisfy the inequalities $a_1(x) \ge a_2(x),\ x\in \OOO$ and 
$F_1(x,t) \ge F_2(x,t), \ (x,t)\in \OOO\times (0,T)$, respectively. 

Then the inequality 
$$
u(F_1, a_1)(x,t) \ge u(F_2,a_2)(x,t), \ (x,t)\in \OOO\times (0,T)
$$
holds true.
\end{corollary}

\begin{proof}
Setting $a:= a_1-a_2$, $F:= F_1 - F_2$ and $u:= u(F_1,a_1) - u(F_2,a_2)$, 
we immediately obtain the inequalities $a(x)\ge 0,\ x\in \OOO$ and $F(x,t)\ge 0, \ (x,t)\in \QQQQ$ and 
$$
\left\{ \begin{array}{rl}
& \pppa (u-a) + Au = F \ge 0 \quad \mbox{in } \QQQQ, \\
& \NUNU u + \sigma u = 0 \quad \mbox{on } \ppp\OOO.
\end{array}\right.
$$
Therefore, Theorem \ref{t2.2} implies that $u(x,t)\ge 0, \ (x,t)\in \OOO\times (0,T)$, that is,
$u(F_1, a_1)(x,t) \ge u(F_2,a_2)(x,t), \ (x,t)\in \OOO\times (0,T)$.
\end{proof}

In its turn, Corollary \ref{c2.1} can be applied for derivation of the lower and upper bounds for the solutions to the initial-boundary value problem \eqref{(2.3)} by  suitably choosing the  initial values and the source functions. Let us  demonstrate this technique on an example.
\begin{example}
\label{ex1}
Let the coefficients $a_{ij}, b_j$, $1\le i,j\le d$   of the operator 
$$
-Av(x) = \sum_{i,j=1}^d \ppp_i(a_{ij}(x)\ppp_jv(x)) + \sum_{j=1}^d
b_j(x,t)\ppp_jv(x)
$$
from the initial-boundary value problem \eqref{(2.3)}  satisfy the conditions \eqref{(1.2)}. Now we consider the homogeneous initial condition $a(x)=0,\ x\in \OOO$  and assume that the source function 
$F \in L^2(0,T;L^2(\OOO))$ satisfies the inequality
$$
F(x,t) \ge \delta t^{\beta}, \quad x\in \OOO,\, 0<t<T
$$
with certain  constants $\beta \ge 0$ and $\delta>0$.

Then the solution $u(F,0)$ can be estimated from below as follows:
\begin{equation}\label{(4.2a)}
u(F,0)(x,t) \ge \frac{\delta\Gamma(\beta+1)}{\Gamma(\alpha+\beta+1)}
t^{\alpha+\beta}, \quad x\in \OOO, \, 0\le t \le T.
\end{equation}
Indeed, it is easy to verify that the function
$$
\underline{u}(x,t):= \frac{\delta\Gamma(\beta+1)}{\Gamma(\alpha+\beta+1)}
t^{\alpha+\beta}, \quad x\in \OOO, \, t>0
$$
is a solution to the following problem:
$$
\left\{\begin{array}{rl}
& \pppa \underline{u} + A\underline{u} = \delta t^{\beta} \quad \mbox{in $\OOO 
\times (0,T)$}, \\
& \ppp_{\nu_A}\underline{u} = 0 \quad \mbox{on $\ppp\OOO \times 
(0,T)$}, \\
& \underline{u}(x,\cdot) \in H_{\alpha}(0,T).
\end{array}\right.
$$
Due to the inequality  $F(x,t) \ge \delta t^{\beta},\ (x,t) \in \OOO \times (0,T)$, we can apply Corollary \ref{c2.1}
to the solutions $u$ and $\underline{u}$ and the inequality  \eqref{(4.2a)} immediately follows.

In particular, for the spatial 
dimensions $d \le 3$,  the Sobolev embedding theorem leads to the inclusion $u  \in L^2(0,T;H^2(\OOO)) \subset 
L^2(0,T;C(\ooo{\OOO}))$ and thus the strict inequality $u(F,0)(x,t) > 0$ holds true for almost all
$t>0$ and all $x\in \ooo{\OOO}$.
\end{example}

Now we proceed to the proof of Theorem \ref{t2.2}.

\begin{proof}
In the proof, we employ the operators $Qv(t)$ and $G(t)$ defined by \eqref{(5.8)}.
In terms of these operators, the solution $u(t):= u(F,a)(t)$ to the initial-boundary problem \eqref{(2.3)} satisfies 
the integral equation
\begin{equation}
\label{(6.1)}
u(F,a)(t) = G(t) + \int^t_0 K(t-s)Qu(s) ds, \quad 0<t<T.   
\end{equation}

For readers' convenience, we split the proof into three parts.
\vspace{0.3cm}

\noindent
I. First part of the proof: existence of a smoother solution.

\vspace{0.3cm}

In the formulation of Lemma \ref{l4.2}, 
we assumed existence of a solution $u\in C([0,T];C^2(\ooo{\OOO}))$
to the initial-boundary value problem \eqref{(4.2)} satisfying the inclusion $t^{1-\alpha}\ppp_tu \in C([0,T];C(\ooo{\OOO}))$.
On the other hand, Theorem \ref{t2.1} asserts the unique existence  
of solution $u$ to the initial-boundary value problem \eqref{(2.3)} from the space $\mathcal{Y}_\alpha$, i.e., of the solution $u$ 
that satisfies the inclusions $u\in L^2(0,T;H^2(\OOO))$ and $u - a 
\in H_{\alpha}(0,T;L^2(\OOO))$.  

In this part of the proof, we show that for $a \in C^{\infty}_0(\OOO)$ and $F \in C^{\infty}_0(\OOO\times
(0,T))$, the  
solution to the problem \eqref{(2.3)} satisfies the regularity assumptions formulated in Lemma 3.

More precisely, we first prove the following lemma:

\begin{lemma}
\label{l6.1}
Let $a_{ij}$, $b_j$, $c$ satisfy the conditions \eqref{(1.2)} and 
the inclusions 
$a\in C^{\infty}_0(\OOO)$, $F \in C^{\infty}_0(\OOO\times (0,T))$ hold true.

Then the solution $u=u(F,a)$ to the problem \eqref{(2.3)} satisfies the inclusions
$$
u \in C([0,T];C^2(\ooo{\OOO})), \quad 
t^{1-\alpha}\ppp_tu \in C([0,T];C(\ooo{\OOO}))
$$
and $\lim_{t\to 0} \Vert u(t) - a\Vert_{L^2(\OOO)} = 0$.
\end{lemma}

\begin{proof}
We recall that $c_0>0$ is a positive fixed constant and 
$$
-A_0v = \sumij \ppp_i(a_{ij}(x)\ppp_jv) - c_0v,\ \DDD(A_0) = \{ v \in H^2(\OOO);\, \NUNU v + \sigma v = 0 \,\,
\mbox{on } \ppp\OOO\}.
$$
Then $\DDD(A_0^{\hhalf}) = H^1(\OOO)$ and
$\Vert A_0^{\hhalf}v\Vert \sim \Vert v\Vert_{H^1(\OOO)}$
(\cite{Fu}).  Moreover, for the operators 
$S(t)$ and 
$K(t)$ defined by \eqref{(5.1)} and \eqref{(5.2)}, the estimates \eqref{(5.3)} hold true.

In what follows, we denote  $ \frac{\ppp u}{\ppp t}(\cdot,t)$ by $u'(t) = \frac{du}{dt}(t)$
if there is no fear of confusion.

The solution $u$ to the integral equation \eqref{(6.1)} can be constructed 
as a fixed point of the equation
\begin{equation}
\label{(6.2)}
A_0u(t) = A_0G(t) + \int^t_0 A_0^{\hhalf}K(t-s)A_0^{\hhalf}Qu(s) ds,
\quad 0<t<T.                                 
\end{equation}

As already proved, this fixed point satisfies the inclusion  
$u\in L^2(0,T;H^2(\OOO))$ $ \cap $ $ (\HH(0,T;L^2(\OOO)) + \{ a\}).$

Now we derive some estimates for the norms 
$\Vert A_0^{\kappa}u(t)\Vert$,
$\kappa=1,2$ and $\Vert A_0u'(t)\Vert$ for $0<t<T$.
First we set 
$$
D:= \sup_{0<t<T} (\Vert A_0F(t)\Vert + \Vert A_0F'(t)\Vert 
+ \Vert A_0^2F(t)\Vert) + \Vert a\Vert_{H^4(\OOO)}.       
$$
Since $F\in C^{\infty}_0(\OOO\times (0,T))$, we obtain the inclusion 
$F\in L^{\infty}(0,T;\DDD(A_0^2))$ and the inequality $D < +\infty$.
Moreover, in view of \eqref{(5.3)}, for $\kappa=1,2$, we get the estimates
\begin{align*}
& \left\Vert A_0^{\kappa}\int^t_0 K(t-s)F(s) ds \right\Vert
\le C\int^t_0 \Vert K(t-s)\Vert \Vert A_0^{\kappa}F(s)\Vert ds\\
\le& C\left( \int^t_0 (t-s)^{\alpha-1}ds \right) 
\sup_{0<s<T} \Vert A_0^{\kappa}F(s)\Vert \le CD,
\end{align*}
\begin{align*}
& \left\Vert A_0\frac{d}{dt}\int^t_0 K(t-s)F(s) ds \right\Vert
= \left\Vert A_0\frac{d}{dt}\int^t_0 K(s)F(t-s) ds \right\Vert\\
=& \left\Vert A_0K(t)F(0) + A_0\int^t_0 K(s)F'(t-s) ds \right\Vert\\
\le& C\left\Vert A_0\int^t_0 K(s)F'(t-s) ds \right\Vert
\le C \int^t_0 s^{\alpha-1}
\Vert A_0F'(t-s) \Vert ds < CD.
\end{align*}
The regularity conditions \eqref{(1.2)} lead to the estimates
$$
\Vert A_0^{\hhalf}Q(s)u(s)\Vert \le C\Vert Q(s)u(s)\Vert_{H^1(\OOO)}
= C\left\Vert \sum_{j=1}^d b_j(s)\ppp_ju(s) + (c_0+c(s))u(s)
\right\Vert_{H^1(\OOO)}
$$
\begin{equation}
\label{(6.4)}
\le C\Vert u(s)\Vert_{H^2(\OOO)} \le C\Vert A_0u(s)\Vert,\quad 
0<s<T.      
\end{equation}
Moreover, 
$$
\Vert A_0S(t)a\Vert = \Vert S(t)A_0a\Vert \le C\Vert a\Vert_{H^2(\OOO)}
\le CD
$$
by using the inequalities \eqref{(5.3)}.  Then 
\begin{align*}
& \Vert A_0u(t)\Vert \le CD 
+ \int^t_0 \Vert A_0^{\hhalf}K(t-s)\Vert \Vert A_0^{\hhalf}Q(s)u(s)\Vert ds\\
\le& CD + C\int^t_0 (t-s)^{\hhalf\alpha -1}\Vert A_0u(s)\Vert ds, \quad
0<s<T.
\end{align*}
The generalized Gronwall inequality yields the estimate
$$
\Vert A_0u(t)\Vert \le CD + C\int^t_0 (t-s)^{\hhalf\alpha -1}D ds
\le CD, \quad 0<t<T,
$$
which implies the inequality
$$
\Vert A_0u\Vert_{L^{\infty}(0,T;H^2(\OOO))} \le CD. 
$$
Next, for the space 
$C([0,T]; L^2(\OOO))$, we can repeat the same arguments  as the ones employed for 
the iterations $R^n$ of the operator $R$ in the proof 
of Theorem \ref{t2.1} and apply the fixed point theorem to the  
equation \eqref{(6.1)} that leads to the inclusion $A_0u \in C([0,T];L^2(\OOO))$.
The obtained results implicate
\begin{equation}
\label{(6.5)}
u\in C([0,T];H^2(\OOO)), \quad 
\Vert u\Vert_{C([0,T];H^2(\OOO))}
\le CD.                               
\end{equation}
Choosing $\ep_0 > 0$ sufficiently small, we have the equation
\begin{equation}
\label{(6.6)}
A_0^{\frac{3}{2}}u(t) = A_0^{\frac{3}{2}}G(t) 
+ \int^t_0 A_0^{\frac{3}{4}+\ep_0} K(t-s)A_0^{\frac{3}{4}-\ep_0}Q(s)u(s)ds,
\quad 0<t<T.                                                
\end{equation}

Next, according to \cite{Fu}, the inclusion
$$
\DDD(A_0^{\frac{3}{4}-\ep_0}) \subset H^{\frac{3}{2}-2\ep_0}(\OOO)
$$
holds true.
Now we proceed to the proof of the inclusion $Q(s)u(s) \in 
\DDD(A_0^{\frac{3}{4}-\ep_0})$.  
By \eqref{(5.3)}, we obtain the inequality
$$
\Vert A_0^{\frac{3}{2}}u(t)\Vert \le CD
+ \int^t_0 (t-s)^{(\frac{1}{4}-\ep_0)\alpha-1} 
\Vert A_0^{\frac{3}{4}-\ep_0}Q(s)u(s)\Vert ds,
$$
which leads to the estimate
\begin{equation}
\label{(6.7a)}
\Vert u(t)\Vert_{H^3(\OOO)} \le CD
+ \int^t_0 (t-s)^{(\frac{1}{4}-\ep_0)\alpha-1} 
\Vert u(s)\Vert_{H^3(\OOO)} ds, \quad 0<t<T
\end{equation}
because of the inequality
$$
\Vert A_0^{\frac{3}{4}-\ep_0}Q(s)u(s)\Vert 
\le C\Vert Q(s)u(s)\Vert_{H^{\frac{3}{2}}(\OOO)}
\le C\Vert Q(s)u(s)\Vert_{H^2(\OOO)}
\le C\Vert u(s)\Vert_{H^3(\OOO)},
$$
which follows from the regularity conditions \eqref{(1.2)} posed on the 
coefficients $b_j, c$. 

For $0<t<T$, the generalized Gronwall inequality applied to the integral inequality \eqref{(6.7a)} yields the estimate 
$$
\Vert u(t)\Vert_{H^3(\OOO)} \le C\left(1 
+ t^{\alpha\left(\frac{1}{4}-\ep_0\right)} 
\right)D.
$$

For the relation \eqref{(6.6)}, we repeat the same arguments  as the ones 
employed in the proof of 
Theorem \ref{t2.1} to estimate 
$A_0^{\frac{3}{2}}u(t)$ in the norm $C([0,T];L^2(\OOO))$ by the fixed point 
theorem arguments and thus we obtain the inclusion
 $A_0^{\frac{3}{2}}u \in C([0,T];L^2(\OOO))$.

Summarising the estimates derived above, we have shown that
\begin{equation}
\label{(6.7)}
\left\{ \begin{array}{rl}
& u \in C([0,T];\mathcal{D}(A_0^{\frac{3}{2}})) \subset C([0,T];H^3(\OOO)), \\
& \Vert u(t)\Vert_{H^3(\OOO)} \le C\left( 
1 + t^{\alpha\left(\frac{1}{4}-\ep_0\right)} \right)D,
\quad 0<t<T.                          
\end{array}\right.
\end{equation}

Next we estimate the norm $\Vert Au'(t)\Vert$.  First,  $u'(t)$ is represented in the form
\begin{align*}
& u'(t) = G'(t) + \frac{d}{dt}\int^t_0 K(t-s)Q(s)u(s) ds\\
= & G'(t) + \frac{d}{dt}\int^t_0 K(s)Q(t-s)u(t-s) ds \, = \, G'(t) + K(t)Q(0)u(0) \\
+ & \int^t_0 K(s) (Q(t-s)u'(t-s)
+ Q'(t-s)u(t-s)) ds, \quad 0<t<T,
\end{align*}
so that  
\begin{equation}
\label{(6.8)}
A_0u'(t) = A_0G'(t) + A_0K(t)Q(0)u(0) 
\end{equation}
$$
+ \int^t_0 A_0^{\hhalf}K(s) A_0^{\hhalf}(Q(t-s)u'(t-s)
+ Q'(t-s)u(t-s)) ds, \quad 0<t<T.       
$$
Similarly to the arguments applied for derivation of \eqref{(6.4)}, we obtain 
the inequality
$$
\Vert A_0^{\hhalf}(Q(t-s)u'(t-s) + Q'(t-s)u(t-s))\Vert 
\le C\Vert A_0u'(t-s)\Vert, \quad 0<t<T.
$$
The inclusion $Q(0)u(0) = Q(0)a \in C^2_0(\OOO) \subset \DDD(A_0)$ follows  
from the regularity conditions \eqref{(1.2)}
and the inclusion $a \in C^{\infty}_0(\OOO)$.  
Furthermore, by \eqref{(5.2)} and \eqref{(5.3)}, we obtain
$$
\Vert A_0S'(t)a\Vert 
= \Vert A_0^2K(t)a\Vert = \Vert K(t)A_0^2a\Vert 
\le Ct^{\alpha-1}\Vert A_0^2a\Vert \le Ct^{\alpha-1}\Vert a\Vert_{H^4(\OOO)}
$$
and 
$$
\Vert K(t)A_0(Q(0)a)\Vert \le Ct^{1-\alpha}\Vert A_0(Q(0)a)\Vert 
\le Ct^{\alpha-1}\Vert a\Vert_{H^3(\OOO)}.
$$
Hence, the representation \eqref{(6.8)} leads to the estimate
$$
\Vert A_0u'(t)\Vert \le Ct^{\alpha-1}D  
+ C\int^t_0 s^{\hhalf\alpha -1}\Vert A_0u'(t-s)\Vert ds,
\quad 0<t<T.
$$
Now we consider a vector space
$$
\www{X}:= \{v\in C([0,T];L^2(\OOO)) \cap C^1((0,T];L^2(\OOO));\, 
t^{1-\alpha}\ppp_tv \in C([0,T];L^2(\OOO))\}
$$
with the norm
$$
\Vert v\Vert_{\www{X}}:= \max_{0\le t\le T} \Vert t^{1-\alpha}\ppp_tv
(\cdot,t)\Vert_{L^2(\OOO)}
+ \max_{0\le t\le T} \Vert v(\cdot,t)\Vert_{L^2(\OOO)}.
$$
It is easy to verify that $\www{X}$ with the norm $\Vert v\Vert_{\www{X}}$ defined above is a Banach space.

Arguing similarly to the proof of Theorem \ref{t2.1} and applying the fixed 
point theorem in the Banach space $\www{X}$, we conclude that 
$A_0u \in \www{X}$, that is, $t^{1-\alpha}A_0u' \in C([0,T];L^2(\OOO))$. 
Using the inclusion  $\DDD(A_0) \subset 
C(\ooo{\OOO})$ in the spatial dimensions $d=1,2,3$,  the Sobolev embedding 
theorem yields 
\begin{equation}
\label{(6.9)}
u' \in C(\ooo{\OOO} \times (0,T]), \quad \Vert A_0u'(t)\Vert 
\le CDt^{\alpha-1}, \quad 0\le t\le T.
\end{equation}

Now we proceed to the estimation of $A_0^2u(t)$.
Since $\frac{d}{ds}(-A_0^{-1}S(s)) = K(s)$ for $0<s<T$ by \eqref{(5.2)}, 
the integration by parts yields 
\begin{align*}
& \int^t_0 K(t-s)Q(s)u(s) ds = \int^t_0 K(s)Q(t-s)u(t-s) ds \\
= & \left[ -A_0^{-1}S(s)Q(t-s)u(t-s)\right]^{s=t}_{s=0}\\
- &
\int^t_0 A_0^{-1}S(s)(Q'(t-s)u(t-s)+Q(t-s)u'(t-s)) ds\\
= & A_0^{-1}Q(t)u(t) - A_0^{-1}S(t)Q(0)u(0)
\end{align*}
\begin{equation}
\label{(6.10)}
- \int^t_0 A_0^{-1}S(s)(Q'(t-s)u(t-s)+Q(t-s)u'(t-s)) ds,
\quad 0<t<T.                        
\end{equation}
Applying the Lebesgue convergence theorem and the estimate
$\vert \MLONE(\eta)\vert \le \frac{C}{1+\eta},\ \eta>0$ 
(Theorem 1.6 in \cite{Po}), we readily reach  
$$
\Vert S(t)a - a\Vert^2
= \sumn \vert (a,\va_n)\vert^2 (\MLONE(-\la_nt^{\alpha}) - 1)^2
\, \longrightarrow\, 0
$$
as $t \to \infty$ for $a \in L^2(\OOO)$.  
\\
Hence,
$u \in C([0,T];L^2(\OOO))$ and $\lim_{t\downarrow 0}
\Vert (S(t)-1)a\Vert = 0$ and thus
$$
\lim_{s\downarrow 0} S(s)Q(t-s)u(t-s) = S(0)Q(t)u(t) \quad
\mbox{in } L^2(\OOO)
$$
and
$$
\lim_{s\uparrow t} S(s)Q(t-s)u(t-s) = S(t)Q(0)u(0) \quad
\mbox{in }L^2(\OOO),
$$
which justify the last equality in the formula \eqref{(6.10)}.

Thus, in terms of \eqref{(6.10)}, the representation \eqref{(5.7)} 
can be rewritten in the form
$$
A_0^2(u(t) - A_0^{-1}Q(t)u(t)) = A_0^2G(t) -A_0S(t)Q(0)u(0) 
$$
\begin{equation}
\label{(6.11)}
- \int^t_0 A_0^{\hhalf}S(s)A_0^{\hhalf}
(Q'(t-s)u(t-s) + Q(t-s)u'(t-s)) ds, \quad
0<t<T.                                       
\end{equation}
Since $u(0) = a \in C^{\infty}_0(\OOO)$ and $F \in C^{\infty}_0(\OOO
\times (0,T))$, in view of \eqref{(1.2)} we have the inclusions
$$
A_0^2G(\cdot) \in C([0,T];L^2(\OOO)), \ 
A_0S(t)Q(0)u(0) = S(t)(A_0Q(0)a) \in C([0,T];L^2(\OOO)).
$$
Now we use the conditions \eqref{(1.2)} and \eqref{(5.3)} and repeat the 
arguments employed for derivation 
of \eqref{(6.4)} by means of \eqref{(6.5)} and \eqref{(6.9)} to obtain the 
estimates
\begin{align*}
& \left\Vert \int^t_0 A_0^{\hhalf}S(s)A_0^{\hhalf} 
(Q'(t-s)u(t-s) + Q(t-s)u'(t-s)) ds \right\Vert\\
\le &C\int^t_0 s^{-\hhalf\alpha}\Vert Q'(t-s)u(t-s) + Q(t-s)u'(t-s)
\Vert_{H^1(\OOO)} ds\\
\le& C\int^t_0 s^{-\hhalf\alpha}(\Vert A_0u'(t-s)\Vert
+ \Vert A_0u(t-s)\Vert) ds
\le Ct^{\hhalf\alpha}D
\end{align*}
and the inclusion
$$
-\int^t_0 A_0^{\hhalf}S(s)A_0^{\hhalf}(Q'(t-s)u(t-s) + Q(t-s)u'(t-s)) ds 
\in C([0,T];L^2(\OOO)).
$$
Therefore, 
$$
A_0^2(u(t) - A_0^{-1}Q(t)u(t)) = A_0(A_0u(t) - Q(t)u(t))
\in C([0,T];L^2(\OOO)),
$$
that is,
$$
A_0u(t) - Q(t)u(t) \in C([0,T]; \DDD(A_0)) \subset C([0,T];H^2(\OOO)).
$$
On the other hand, the estimate \eqref{(6.7)} implies  $Q(t)u(t) \in C([0,T];H^2(\OOO))$
and we obtain 
\begin{equation}
\label{(6.12a)}
A_0u(t) \in C([0,T];H^2(\OOO)).   
\end{equation}

For further arguments, we define the Schauder spaces $C^{\theta}(\ooo{\OOO})$ 
and 
$C^{2+\theta}(\ooo{\OOO})$ with $0<\theta<1$ (see, e.g., \cite{GT},
\cite{LU}) as follows:  A function $w$ is said to belong to the space 
$C^{\theta}(\ooo{\OOO})$ if 
$$
\sup_{x, x'\in \OOO, \, x \ne x'} 
\frac{\vert w(x) - w(x')\vert}{\vert x-x'\vert^{\theta}}
< \infty.
$$  
For  $w \in C^{\theta}(\ooo{\OOO})$, we define the norm
$$
\Vert w\Vert_{C^{\theta}(\ooo{\OOO})}
:= \Vert w\Vert_{C(\ooo{\OOO})}
+ \sup_{x, x'\in \OOO, \, x \ne x'}  
\frac{\vert w(x) - w(x')\vert}{\vert x-x'\vert^{\theta}}
$$
and for $w\in C^{2+\theta}(\ooo{\OOO})$, the norm is given by
$$
\Vert w\Vert_{C^{2+\theta}(\ooo{\OOO})}
:= \Vert w\Vert_{C^2(\ooo{\OOO})}
+ \sum_{\vert \tau\vert=2}
\sup_{x, x'\in \OOO, \, x \ne x'} 
\frac{\vert \ppp_x^{\tau}w(x) - \ppp_x^{\tau}w(x')\vert}
{\vert x-x'\vert^{\theta}}.
$$

In the last formula, the notations 
 $\tau := (\tau_1, ..., \tau_d) \in (\N \cup \{0\})^d$,
$\ppp_x^{\tau}:= \ppp_1^{\tau_1}\cdots \ppp_d^{\tau_d}$, and
$\vert \tau\vert:= \tau_1 + \cdots + \tau_d$ are employed.

For $d=1,2,3$, the Sobolev embedding theorem says that $H^2(\OOO) \subset 
C^{\theta}(\ooo{\OOO})$ with some $\theta \in (0,1)$
(\cite{Ad}).  

Therefore, in view of \eqref{(6.12a)}, we obtain the inclusion 
$h:= A_0u(\cdot,t) 
\in C^{\theta}(\ooo{\OOO})$ for each $t \in [0,T]$.
Now we apply the Schauder 
estimate (see, e.g., \cite{GT} or \cite{LU})
for solutions to the elliptic boundary value problem
$$
A_0u(\cdot,t) = h\in C^{\theta}(\ooo{\OOO}) \quad \mbox{in } \OOO
$$
with the boundary condition $\NUNU u(\cdot,t) + \sigma(\cdot)u(\cdot,t) = 0$ 
on $\ppp\OOO$ to reach the inclusion
$$
u \in C([0,T]; C^{2+\theta}(\ooo{\OOO})).     
$$
This inclusion and \eqref{(6.9)} yield the conclusion 
$u \in C([0,T];C^2(\ooo{\OOO}))$ and 
$t^{1-\alpha}\ppp_tu \in C([0,T];C(\ooo{\OOO}))$ of the lemma.

Finally we prove that $\lim_{t\to 0} \Vert u(t) - a \Vert = 0$.
By \eqref{(5.3)}, we have
\begin{align*}
& \left\Vert \int^t_0 K(t-s)h(s) ds\right\Vert
\le \int^t_0 \Vert K(t-s)h(s) \Vert ds 
\le C\int^t_0 (t-s)^{\alpha-1} \Vert h(s)\Vert ds\\
\le & \frac{Ct^{\alpha}}{\alpha}\Vert h\Vert_{L^{\infty}(0,T;L^2(\OOO))},
\end{align*}
and so
\begin{equation}\label{(6.13)}
\lim_{t\to 0} \int^t_0 K(t-s)h(s) ds = 0 \quad \mbox{in $L^2(\OOO)$}
\end{equation}
for each $h \in L^{\infty}(0,T;L^2(\OOO))$.
Therefore by the regularity $u \in C([0,T];C^2(\ooo{\OOO}))$, we see that 
$$
\lim_{t\to 0} \left( \int^t_0 K(t-s)F(s) ds + Ru(t) \right)
= 0 \quad \mbox{in $L^2(\OOO)$},
$$
where $R$ is defined in \eqref{(5.8)}.
Moreover, for justifying \eqref{(6.10)}, we have already proved
$\lim_{t\to 0} \Vert S(t)a - a\Vert = 0$ for $a \in L^2(\OOO)$.
Thus the proof of Lemma \ref{l6.1} is complete.
\end{proof}

\vspace{0.3cm}

\noindent
II. Second part of the proof.

\vspace{0.3cm}

In this part, we weaken the regularity conditions posed on the solution $u$ 
to \eqref{(4.2)} in 
Lemma \ref{l4.2} and prove the same
results provided that $u\in L^2(0,T;H^2(\OOO))$ and 
$u-a \in \HH(0,T;L^2(\OOO))$, under the assumption that 
$\min\limits_{(x,t)\in \ooo{\OOO}\times [0,T]} b_0(x,t) > 0$ 
is sufficiently large.

Let $F \in L^2(0,T;L^2(\OOO))$ and $a\in H^1(\OOO)$ satisfy the inequalities 
 $F(x,t)\ge 0,\ (x,t)\in \OOO\times (0,T)$ and $a(x)\ge 0,\ x\in \OOO$.

Now we apply the standard mollification procedure (see, e.g., \cite{Ad}) 
and construct the sequences
$F_n \in C^{\infty}_0(\OOO\times (0,T))$ and $a_n \in C^{\infty}_0(\OOO)$,
$n\in \N$ such that $F_n(x,t)\ge 0,\ (x,t)\in \OOO\times (0,T)$ and $a_n(x)\ge 0,\ x\in \OOO$, $n\in \N$ and $\lim_{n\to\infty}
\Vert F_n-F\Vert_{L^2(0,T;L^2(\OOO))} = 0$ and 
$\lim_{n\to\infty}\Vert a_n-a\Vert_{H^1(\OOO)} = 0$.
Then Lemma \ref{l6.1} yields the inclusion
$$
u(F_n,a_n) \in C([0,T];C^2(\ooo{\OOO})), \quad
t^{1-\alpha}\ppp_tu(F_n,a_n) \in C([0,T];C(\ooo{\OOO})),
\quad n\in \N
$$
and thus Lemma \ref{l4.2} ensures the inequalities
\begin{equation}
\label{(6.14)}
u(F_n,a_n)(x,t) \ge 0 ,\ \ (x,t)\in  \OOO\times (0,T), \, n\in \N.
\end{equation}
Since Theorem \ref{t2.1} holds true  for the 
initial-boundary value problem \eqref{(4.2)} with $F$ and $a$ 
replaced by $F-F_n$ and $a-a_n$, respectively, we have 
$$
\Vert u(F,a) - u(F_n,a_n)\Vert_{L^2(0,T; H^2(\OOO))}
$$
$$ 
\le C(\Vert a-a_n\Vert_{H^1(\OOO)} 
+ \Vert F-F_n\Vert_{L^2(0,T;L^2(\OOO))}) \, \to \, 0
$$
as $n\to \infty$.  Therefore, we can choose a subsequence $m(n)\in \N$
such that $u(F,a)(x,t) = \lim_{m(n)\to \infty} u(F_{m(n)},a_{m(n)})(x,t)$
for almost all $(x,t) \in \OOO\times (0,T)$.
Then the inequality \eqref{(6.14)} leads to the desired result, namely, to the inequality   
$u(F,a)(x,t) \ge 0$ for almost all 
$(x,t) \in \OOO\times (0,T)$.

\vspace{0.3cm}

\noindent
III. Third part of the proof.

\vspace{0.3cm}

Let the inequalities $a(x)\ge 0,\ x\in \OOO$ and  $F(x,t)\ge 0,\ (x,t)\in \OOO\times (0,T)$ hold true for $a\in H^1(\OOO)$ and $F \in L^2(0,T;L^2(\OOO))$  and let $u=u(F,a) \in 
L^2(0,T;H^2(\OOO))$ is a solution to the problem \eqref{(2.3)}.
In order to complete the proof of Theorem 2, we have to demonstrate  
the non-negativity of the solution without any assumptions
on the sign of the zeroth-order coefficient.

First, the zeroth-order coefficient $b_0(x,t)$ in the definition \eqref{(3.1a)} of the operator $-A_1$ is set to a constant $b_0>0$ that is
assumed to be sufficiently large.
In this case, the initial-boundary value problem \eqref{(2.3)}
can be rewritten as follows:
\begin{equation}
\label{(6.15)}
\left\{ \begin{array}{rl}
& \pppa (u-a) + A_1u = (b_0+c(x,t))u + F(x,t), \quad 
(x,t) \in \OOO\times (0,T), \\
& \NUNU u + \sigma u = 0 \quad \mbox{on $\ppp\OOO\times (0,T)$}.
\end{array}\right.
\end{equation}
In what follows, we choose sufficiently large $b_0>0$ such that 
$b_0 \ge \Vert c\Vert_{C(\ooo{\OOO} \times [0,T])}$.

In the previous parts of the proof, we already interpreted the solution $u$ as 
a unique fixed point for the equation \eqref{(6.1)}. 
Now let us construct an appropriate approximating sequence 
$u_n$, $n\in \N$ for the fixed point $u$. First we set $u_0(x,t) := 0$ for $(x,t) \in \OOO\times (0,T)$ and $u_1(x,t) = a(x) \ge 0, \ (x,t) \in \OOO\times (0,T)$. Then 
we define a sequence $u_{n+1},\ n\in \N$ 
of solutions to the following initial-boundary value problems with the given $u_n$:
\begin{equation}
\label{(6.16)}
\left\{ \begin{array}{rl}
&\pppa (u_{n+1}-a) + A_1u_{n+1} = (b_0+c(x,t))u_n + F(x,t)
\quad \mbox{in } \OOO\times (0,T),\\
& \NUNU u_{n+1} + \sigma u_{n+1} = 0 \quad \mbox{on } \ppp\OOO\times (0,T),\\
& u_{n+1} - a \in \HH(0,T;L^2(\OOO)), \quad n\in \N.
\end{array}\right.
\end{equation}

First we show that
\begin{equation}
\label{(6.17)}
u_n(x,t) \ge 0, \quad (x,t) \in \OOO\times (0,T), \quad n\in \N.
\end{equation}
Indeed, the inequality \eqref{(6.17)} holds for $n=1$.  Now we assume that 
$u_n(x,t) \ge 0,\ (x,t)\in \OOO\times (0,T)$.
Then $(b_0+c(x,t))u_n(x,t) + F(x,t) \ge 0,\ (x,t)\in \QQQQ$, and thus by the results 
established in the second part of the proof of Theorem \ref{t2.2}, 
we obtain the inequality $u_{n+1}(x,t) \ge 0,\ (x,t)\in \QQQQ$.  
By the principle of mathematical induction, 
the inequality \eqref{(6.17)} holds true for all $n\in \N$.

Now we rewrite the problem \eqref{(6.16)} as
$$
\pppa (u_{n+1}(t) - a) + A_0u_{n+1}(t) 
= (Q(t)u_{n+1}-(c(t)+b_0)u_{n+1})
+ (b_0+c(t))u_n + F,
$$
where $A_0$ and $Q(t)$ are defined by \eqref{(3.2)} and \eqref{(5.8)}, respectively. 
Next we estimate $w_{n+1}:= u_{n+1} - u_n$.  By the relation \eqref{(6.16)}, 
$w_{n+1}$ is a solution to the problem
$$
\left\{ \begin{array}{rl}
&\pppa w_{n+1} + A_0w_{n+1} = (Q(t)w_{n+1} - (c(t)+b_0)w_{n+1})
+ (b_0+c(x,t))w_n \\ 
& \qquad \qquad \quad \mbox{in } \OOO\times (0,T),\\
& \NUNU w_{n+1} + \sigma w_{n+1} = 0 \quad \mbox{on } \ppp\OOO\times (0,T),\\
& w_{n+1} \in \HH(0,T;L^2(\OOO)), \quad n\in \N.
\end{array}\right.
$$
In terms of the operator $K(t)$ defined by \eqref{(5.2)}, acting  
similarly to our analysis of the fixed point equation \eqref{(6.1)}, 
we obtain the integral equation 
\begin{align*}
& w_{n+1}(t) = \int^t_0 K(t-s)(Qw_{n+1})(s) ds
- \int^t_0 K(t-s)(c(s)+b_0)w_{n+1}(s) ds\\
+ & \int^t_0 K(t-s)(b_0+c(s))w_n(s) ds, \quad 0<t<T, 
\end{align*}
which leads to the inequalities  
\begin{align*}
& \Vert A_0^{\hhalf}w_{n+1}(t)\Vert 
\le \int^t_0 \Vert A_0^{\hhalf}K(t-s)\Vert 
\Vert Q(s)w_{n+1}(s)\Vert ds\\
+& \int^t_0 \Vert A_0^{\frac{1}{2}}K(t-s)\Vert
\Vert (c(s)+b_0)w_{n+1}(s)\Vert ds
+ \int^t_0 \Vert A_0^{\hhalf}K(t-s)\Vert \Vert (b_0+c(s))w_n(s)\Vert ds\\
\le& C\int^t_0 (t-s)^{\hhalf\alpha-1}
\Vert A_0^{\hhalf}w_{n+1}(s)\Vert ds 
+ C\int^t_0 (t-s)^{\hhalf\alpha-1}\Vert A_0^{\hhalf}w_n(s)\Vert ds
\quad \mbox{for $0<t<T$.}
\end{align*}
For their derivation, we used the norm estimates
$$
\Vert Q(s)w_{n+1}(s)\Vert 
\le C\Vert w_{n+1}(s)\Vert_{H^1(\OOO)} 
\le C\Vert A_0^{\hhalf}w_{n+1}(s)\Vert
$$
and
$$
\Vert (c(s)+b_0)w_\ell(s)\Vert \le C\Vert w_\ell(s)\Vert_{H^1(\OOO)}
\le C\Vert A_0^{\hhalf}w_\ell(s)\Vert, \quad \ell=n, n+1
$$
that hold true under the conditions \eqref{(1.2)}.
Thus  we arrive at the integral inequality 
\begin{align*}
& \Vert A_0^{\hhalf}w_{n+1}(t)\Vert 
\le C\int^t_0 (t-s)^{\frac{1}{2}\alpha-1} 
\Vert A_0^{\hhalf}w_{n+1}(s)\Vert ds\\
+ & C\int^t_0 (t-s)^{\frac{1}{2}\alpha-1} 
\Vert A_0^{\hhalf}w_n(s)\Vert ds, \quad 0<t<T.
\end{align*}
The generalized Gronwall inequality yields now the estimate 
\begin{align*}
&\Vert \Ahalf w_{n+1}(t)\Vert
\le C\int^t_0 (t-s)^{\hhalf\alpha-1}\Vert \Ahalf w_n(s)\Vert ds\\
+& C\int^t_0 (t-s)^{\hhalf\alpha-1}
\left( \int^s_0 (s-\xi)^{\hhalf\alpha-1}
\Vert \Ahalf w_n(\xi)\Vert d\xi\right)ds.
\end{align*}
The  second term at the right-hand side of the last inequality can be represented as follows: 
\begin{align*}
& \int^t_0 (t-s)^{\hhalf\alpha-1}
\left( \int^s_0 (s-\xi)^{\hhalf\alpha-1}\Vert \Ahalf w_n(\xi)\Vert d\xi\right)
ds\\
=& \int^t_0 \Vert \Ahalf w_n(\xi)\Vert \left( \int^t_{\xi}
(t-s)^{\hhalf\alpha-1} (s-\xi)^{\hhalf\alpha-1} ds \right) d\xi\\
=& \frac{\Gamma\left( \hhalf\alpha\right)\Gamma\left( \hhalf\alpha\right)}
{\Gamma(\alpha)}\int^t_0 (t-\xi)^{\alpha-1}\Vert \Ahalf w_n(\xi)\Vert d\xi\\
=& \frac{\Gamma\left( \hhalf\alpha\right)^2}{\Gamma(\alpha)}T^{\hhalf\alpha}
\int^t_0 (t-s)^{\hhalf\alpha-1}\Vert \Ahalf w_n(s)\Vert ds.
\end{align*}
Thus, we can choose a constant $C>0$ depending on $\alpha$ and $T$, 
such that 
\begin{equation}
\label{(ineq1)}
\Vert \Ahalf w_{n+1}(t)\Vert 
\le C\int^t_0 (t-\xi)^{\hhalf\alpha-1}\Vert \Ahalf w_n(s)\Vert ds,
\quad 0<t<T,\, n\in \N.
\end{equation}
Recalling that 
$$
\int^t_0 (t-s)^{\hhalf\alpha -1}\eta(s)ds = \Gamma\left( \hhalf\alpha\right)
(J^{\hhalf\alpha}\eta)(t), \quad t>0,
$$
and setting $\eta_n(t):= \Vert \Ahalf w_n(t)\Vert$, we can rewrite 
\eqref{(ineq1)} in the form
\begin{equation}
\label{(6.18)}
\eta_{n+1}(t) \le C\Gamma\left( \hhalf\alpha\right)(J^{\hhalf\alpha}\eta_n)(t),
\quad 0<t<T, \, n\in \N.                        
\end{equation}

Since the Riemann-Liouville integral $J^{\hhalf\alpha}$ preserves the sign  and the semi-group property 
$J^{\beta_1}(J^{\beta_2}\eta)(t) = J^{\beta_1+\beta_2}\eta(t)$ is valid for 
any $\beta_1, \beta_2 > 0$, applying the inequality \eqref{(6.18)} 
repeatedly, we obtain the estimates
\begin{align*}
& \eta_n(t) \le \left(C\Gamma\left( \hhalf\alpha\right)\right)^{n-1}
(J^{(n-1)\frac{\alpha}{2}}\eta_1)(t) \\
= & \frac{\left(C\Gamma\left( \hhalf\alpha\right)\right)^{n-1}}
{\Gamma\left( \frac{\alpha}{2}(n-1)\right)}
\left( \int^t_0 (t-s)^{(n-1)\hhalf\alpha-1} ds\right) \Vert \Ahalf a\Vert\\
= & \frac{\left(C\Gamma\left( \hhalf\alpha\right)\right)^{n-1}}
{\Gamma\left( \frac{\alpha}{2}(n-1)\right)}
\frac{t^{(n-1)\hhalf\alpha}}{(n-1)\hhalf\alpha} \Vert \Ahalf a\Vert
\le C_1\frac{\left(C\Gamma\left( \hhalf\alpha\right) T^{\frac{\alpha}{2}}
\right)^{n-1}}
{\Gamma\left( \frac{\alpha}{2}(n-1)\right)}.
\end{align*}
The known asymptotic behavior of the gamma function justifies the relation
$$
\lim_{n\to \infty} 
\frac{\left(C\Gamma\left( \hhalf\alpha\right) T^{\frac{\alpha}{2}}
\right)^{n-1}}
{\Gamma\left( \frac{\alpha}{2}(n-1)\right)} = 0.
$$
Thus we have proved that the sequence $u_N = w_0 + \cdots + w_N$ converges to 
the solution 
$u$ in $L^{\infty}(0,T;H^1(\OOO))$ as $N \to \infty$.
Therefore, we can choose a subsequence $m(n)\in \N$
such that $\lim_{m(n)\to\infty} u_{m(n)}(x,t) = u(x,t)$ for 
almost all $(x,t) \in \QQQQ$.
This statement in combination with the inequality \eqref{(6.17)} means that 
$u(x,t) \ge 0$ for almost all
$(x,t) \in \QQQQ$.  The proof of Theorem \ref{t2.2} is completed.
\end{proof}

Now let us fix a source function $F = F(x,t) \ge 0,\ (x,t)\in \QQQQ$ and an initial 
value $a \in H^1(\OOO)$ in the initial-boundary value problem \eqref{(2.3)}  
and denote 
by $u(c,\sigma) = u(c,\sigma)(x,t)$  the solution 
to the  problem \eqref{(2.3)} with
the functions $c=c(x,t)$ and $\sigma = \sigma(x)$.
Then the following comparison property regarding the coefficients 
$c$ and $\sigma$ is valid:

\begin{theorem}
\label{t2.3}
Let $a\in H^1(\OOO)$ and $F \in L^2(\OOO\times (0,T))$ and the inequalities  $a(x)\ge 0,\ x\in \OOO$ and
$F(x,t)\ge 0,\ (x,t)\in \OOO \times (0,T)$ hold true.
\\
(i) Let $c_1, c_2 \in C^1([0,T]; C^1(\ooo{\OOO})) \cap
C([0,T];C^2(\ooo{\OOO}))$ and
$c_1(x,t) \ge c_2(x,t)$ for $(x,t)\in \OOO$.
Then $u(c_1,\sigma)(x,t) \ge u(c_2,\sigma)(x,t)$ in $\OOO \times (0,T)$.
\\
(ii) Let $c(x,t) < 0,\ (x,t) \in \QQQQ$ and a constant $\sigma_0>0$ be arbitrary and fixed.
If the smooth functions $\sigma_1, \sigma_2$ on $\ppp\OOO$ satisfy  the conditions 
$$
\sigma_2(x) \ge \sigma_1(x) \ge \sigma_0,\ x\in \ppp\OOO,
$$
then the inequality $u(c,\sigma_1) \ge u(c, \sigma_2),\ x\in \QQQQ$ holds true.
\end{theorem}

\begin{proof}
We start with a proof of the statement (i). Because $a(x)\ge 0,\ x\in\OOO$ and $F(x,t)\ge 0,\ (x,t)\in \QQQQ$, Theorem \ref{t2.2}
yields the inequality
$u(c_2,\sigma)(x,t)\ge 0,\ (x,t)\in \QQQQ$.
Setting $u(x,t):= u(c_1,\sigma)(x,t) - u(c_2,\sigma)(x,t)$ for $(x,t) 
\in \QQQQ$, we obtain
$$
\left\{ \begin{array}{rl}
& \pppa u - \sumij \ppp_i(a_{ij}\ppp_ju) 
- \sum_{j=1}^d b_j \ppp_ju\\
-& c_1(x,t)u = (c_1-c_2)u(c_2,\sigma)(x,t) \quad \mbox{in }\QQQQ, \\
& \NUNU u + \sigma u = 0 \quad \mbox{on } \ppp\OOO,\\
& u \in \HH(0,T;L^2(\OOO)).
\end{array}\right.
$$
Since $u(c_2,\sigma)(x,t) \ge 0$ and $(c_1-c_2)(x,t) \ge 0$ for $(x,t) 
\in \QQQQ$,
 Theorem \ref{t2.2} leads to the estimate $u(x,t)\ge 0$ for $(x,t) \in \QQQQ$, 
which is equivalent to the inequality $u(c_1,\sigma)(x,t) 
\ge u(c_2,\sigma)(x,t)$ for $(x,t) \in \QQQQ$ and the statement (i) is proved.
\vspace{0.2cm}

Now we proceed to the proof of the statement (ii). Similarly to the procedure applied for the second part of the proof of 
Theorem \ref{t2.2}, we choose the 
sequences $F_n \ge 0$, $F_n \in C^{\infty}_0(\QQQQ)$ and 
$a_n \ge 0$, $a_n \in C^{\infty}_0(\OOO)$, $n\in \N$ such that 
$F_n \to F$ in $L^2(\QQQQ)$ and
$a_n \to a$ in $H^1(\OOO)$.  Let $u_n$, $v_n$ be the 
solutions to the initial-boundary value problem \eqref{(2.3)} with 
$F=F_n$, $a=a_n$ and with the coefficients 
$\sigma_1$ and $\sigma_2$ in the boundary condition,
respectively.  According to Lemma \ref{l6.1}, the inclusions $v_n, u_n \in C(\ooo{\OOO} \times [0,T])$
and $t^{1-\alpha}\ppp_tv_n, \, t^{1-\alpha}\ppp_tu_n 
\in C([0,T];C(\ooo{\OOO}))$, $n\in \N$ hold true 
and thus Theorem \ref{t2.2} yields  
\begin{equation}\label{(4.22a)}
v_n(x,t) \ge 0, \quad (x,t) \in \ppp\OOO\times (0,T).
\end{equation} 

Moreover,  the relation
\begin{equation}
\label{(6.19)}
\lim_{n\to\infty}\Vert u_n - u(c,\sigma_1)\Vert_{L^2(0,T;L^2(\OOO))}
= \lim_{n\to\infty}\Vert v_n - u(c,\sigma_2)\Vert_{L^2(0,T;L^2(\OOO))}
= 0                                        
\end{equation}
follows from Theorem \ref{t2.1}.
Let us now define an auxiliary function $w_n:= u_n - v_n$.
For this function, the inclusions 
\begin{equation}\label{(4.23)}
t^{1-\alpha}\ppp_tw_n \in C([0,T];C(\ooo{\OOO})), \quad 
w_n \in C([0,T];C^2(\ooo{\OOO})), \quad n\in \N
\end{equation}
hold true. Furthermore, it is a solution to the initial-boundary value problem 
\begin{equation}\label{(4.24)}
\left\{ \begin{array}{rl}
& \pppa w_n + Aw_n = 0 \quad \mbox{in } \QQQQ,\\
& \NUNU w_n + \sigma_1w_n = (\sigma_2-\sigma_1)v_n
  \quad \mbox{on } \ppp\OOO\times (0,T),\\
& w_n(x,\cdot) \in \HH(0,T) \quad \mbox{for almost all } x\in \OOO.
\end{array}\right.
\end{equation}
The inequalities \eqref{(4.22a)} and $\sigma_2(x) \ge \sigma_1(x), \ x\in \ppp\OOO$ lead to the  
estimate
\begin{equation}\label{(4.23a)}
\NUNU w_n + \sigma_1w_n \ge 0 \quad \mbox{on $\ppp\OOO\times (0,T)$}.
\end{equation}

To finalize the proof of the theorem,  a variant of Lemma \ref{l4.2} formulated below will be employed.

\begin{lemma}
\label{l4.2a}
Let the elliptic operator $-A$ be defined by \eqref{(2.1)} and the conditions 
 \eqref{(1.2)} be satisfied. Moreover, let the inequality $c(x,t) < 0$ for $x \in \ooo{\OOO}$ and
$0\le t \le T$ hold true and there exist a constant $\sigma_0>0$ such that 
$$
\sigma(x) \ge \sigma_0 \quad \mbox{for all $x\in \ppp\OOO$}.
$$
For $a \in H^1(\OOO)$ and $F\in L^2(\OOO\times (0,T))$, we further assume
that there exists a solution $u\in C([0,T];C^2(\ooo{\OOO}))$ 
 to the initial-boundary value problem 
$$
\left\{ \begin{array}{rl}
& \pppa (u-a) + Au = F \quad \mbox{in $\OOO\times (0,T)$}, \\
& \ppp_{\nu_A}u + \sigma(x)u \ge 0 \quad \mbox{on 
$\ppp\OOO \times (0,T)$}, \\
& u(x,\cdot) - a\in H_{\alpha}(0,T) \quad \mbox{for almost all
$x\in \OOO$}
\end{array}\right.
$$
that satisfies the inclusion $t^{1-\alpha}\ppp_tu \in C([0,T];C(\ooo{\OOO}))$.
Then the inequalities  $F(x,t) \ge 0,\ (x,t)\in \OOO \times (0,T)$ and $a(x)\ge 0, \  \OOO$ implicate the inequality 
$u(x,t) \ge 0,\ (x,t)\in \OOO\times (0,T)$.
\end{lemma}

In the formulation of this lemma,  at the expense of the extra condition $\sigma(x) > 0$
on $\ppp\OOO$, we do not assume that $\min\limits_{(x,t)\in 
\ooo{\OOO}\times [0,T]} (-c(x,t))$ is sufficiently large. This is    the main 
difference between the conditions supposed in Lemma \ref{l4.2a} and in Lemma \ref{l4.2}.
The proof of Lemma \ref{l4.2a} is much simpler compared to the one of Lemma \ref{l4.2}; it will be presented at the end of this section.

Now we complete the proof of Theorem \ref{t2.3}.
Since $c(x,t) < 0$ for $(x,t)\in \QQQQ$ and $\sigma_1(x) \ge \sigma_0 > 0$
on $\ppp\OOO$ and taking into account the conditions \eqref{(4.23)} and \eqref{(4.23a)},  we can apply 
Lemma \ref{l4.2a} to the initial-boundary value problem  \eqref{(4.24)} and deduce the inequality $w_n(x,t) \ge 0,\ (x,t)\in \QQQQ$, 
that is,
$u_n(x,t) \ge v_n(x,t),\ (x,t)\in \QQQQ$ for $n\in \N$.  Due to the  relation \eqref{(6.19)}, 
we can choose
a suitable subsequence of $w_n,\ n\in \N$ and pass to the limit as $n$ tends to infinity
thus arriving at the inequality  
$u(c,\sigma_1)(x,t) \ge u(c,\sigma_2)(x,t)$ in $\QQQQ$.
The proof of Theorem \ref{t2.3} is completed.
\end{proof}

At this point, let us mention a direction for further research in connection with 
the results formulated and proved in this sections. In order to remove the 
negativity condition posed on the coefficient $c=c(x,t)$ in Theorem \ref{t2.3} 
(ii),
one needs a unique existence result for  solutions to the initial-boundary 
value problems  of type \eqref{(2.3)}
with non-zero Robin boundary condition similar to the one formulated in Theorem 
\ref{t2.1}.
There are several works that treat the case of the initial-boundary value 
problems  with non-homogeneous Dirichlet boundary conditions
(see, e.g., \cite{Ya18} and the references therein).
However, to the best of the authors' knowledge, analogous  results are not 
available for  the 
initial-boundary value problems with 
the non-homogeneous Neumann or Robin boundary conditions.  Thus, in Theorem \ref{t2.3} (ii), 
we assumed the condition 
$c(x,t)<0,\ (x,t)\in \QQQQ$, although our conjecture is that this result holds true for 
an arbitrary coefficient $c=c(x,t)$.

We conclude this section with 
a  proof of Lemma \ref{l4.2a} that  is simple because in this case we do not need the function 
$\psi$ defined as in \eqref{(4.3)}.

\begin{proof}
First we introduce an auxiliary function as follows:
$$
\www{w}(x,t):= u(x,t) + \ep(1+t^{\alpha}), \quad x\in \OOO,\,
0<t<T.
$$

The inequalities $c(x,t)<0,\ (x,t)\in \ooo{\OOO} \times [0,T]$ and $\sigma(x) \ge \sigma_0>0,\ x\in \ppp\OOO$ and the calculations similar  to the ones done in the proof of Lemma \ref{l4.2} implicate the inequalities 
$$
\ddda \www{w} + A\www{w} = F + \ep\Gamma(\alpha+1)
- c(x,t)\ep(1+t^{\alpha}) > 0 \quad \mbox{in $\OOO\times (0,T)$},
$$
$$
\ppp_{\nu_A}\www{w} + \sigma \www{w} = \ppp_{\nu_A}u + \sigma u 
+ \sigma\ep(1+t^{\alpha}) \ge \sigma_0\ep
\quad \mbox{on $\ppp\OOO \times (0,T)$}
$$
and
$$
\www{w}(x,0) = a(x) + \ep \ge \ep \quad \mbox{in $\OOO$}.
$$
Based on these inequalities,  the same arguments that were employed after the formula \eqref{(4.7)} in the proof of
Lemma \ref{l4.2} readily complete the proof of Lemma \ref{l4.2a}.
\end{proof}
%
%

\section{Appendix}
\label{sec8}
\setcounter{section}{5}
\setcounter{equation}{0}

In the proof of Lemma \ref{l4.2} that is a basis for all other derivations presented in this paper, we essentially used  an auxiliary function  that satisfies the conditions \eqref{(4.3)}. Thus, ensuring existence of such function is an important problem worth for detailed considerations. In this Appendix, we present a solution to this problem. 

For the readers' convenience, we split our existence proof  into three parts.

\vspace{0.3cm}

\noindent
I. First part of the proof.

\vspace{0.3cm}

In this part, we prove the following lemma:

\begin{lemma}
\label{lem1}
Let the conditions \eqref{(1.2)} be satisfied and the constant 
$$
M:= \min_{(x,t)\in \ooo{\OOO}\times [0,T]} b_0(x,t)>0
$$ 
is sufficiently large.

Then there exists a constant $\kappa_1>0$ such that
\begin{equation}
\label{1}
(A_1(t)v,\, v) \ge \kappa_1\Vert v\Vert_{H^1(\OOO)}^2
\end{equation} 
for all $v \in H^2(\OOO)$ satisfying 
$\ppp_{\nu_A}v + \sigma v = 0$ on $\ppp\OOO$ for each $t \in [0,T]$.
\end{lemma}

In particular, Lemma \ref{lem1} implies that all of the eigenvalues of the operator $A_0$ 
defined by \eqref{(3.2)} are positive if the constant $c_0>0$ is sufficiently large.
Henceforth we employ the notation $b=(b_1,..., b_d)$.

\begin{proof}
By using the conditions \eqref{(1.2)} and the boundary condition  $\NUNU v + \sigma v = 0$ on $\ppp\OOO$, 
integration by parts yields
\begin{align*}
& (A_1(t)v,v)\\
= &-\int_{\OOO} \sum_{i,j=1}^d \ppp_i(a_{ij}(x)\ppp_jv)v dx
 - \frac{1}{2}\int_{\OOO} \sum_{j=1}^d b_j(x,t)\ppp_j(\vert v\vert^2) dx
+ \int_{\OOO} b_0(x,t) \vert v\vert^2 dx\\
=& \int_{\OOO} \sum_{i,j=1}^d a_{ij}(x)(\ppp_iv)(\ppp_jv) dx
- \int_{\ppp\OOO} (\ppp_{\nu_A}v)v dS\\
+ & \frac{1}{2}\int_{\OOO} (\mbox{div}\, b)\vert v\vert^2 dx
- \frac{1}{2}\int_{\ppp\OOO} (b\cdot \nu)\vert v\vert^2 dS
+ \int_{\OOO} b_0(x,t) \vert v\vert^2 dx\\
\ge & \kappa \int_{\OOO} \vert \nabla v\vert^2 dx 
+ \int_{\OOO} \left(\min_{(x,t)\in\ooo{\OOO}\times [0,T]} b_0(x,t)
- \frac{1}{2}\vert \mbox{div}\, b\vert \right)\vert v\vert^2 dx\\
+& \int_{\ppp\OOO} \left( \sigma - \frac{1}{2}(b\cdot \nu)\right)
\vert v\vert^2 dS
\end{align*}
\begin{equation}\label{(2)}
\ge \kappa \int_{\OOO} \vert \nabla v\vert^2 dx 
+ \left( M - \frac{1}{2}\Vert \mbox{div}\, b\Vert_{C(\ooo{\OOO}\times [0,T])}
\right)\int_{\OOO} \vert v\vert^2 dx - C\int_{\OOO} \vert v\vert^2 dS.
\end{equation}
Here and henceforth $C>0$, $C_{\ep}, C_{\delta} > 0$, etc. denote generic 
constants which are independent of the function $v$.  

By the trace theorem (Theorem 9.4 (pp. 41-42) in \cite{LM}), for
$\delta \in (0, \frac{1}{2}]$, there exists
a constant $C_{\delta}>0$ such that 
$$
\Vert v\Vert_{L^2(\ppp\OOO)} \le C_{\delta}\Vert v\Vert
_{H^{\delta+\frac{1}{2}}(\OOO)} \quad \mbox{for all
$v \in H^1(\OOO)$.}
$$
Now we fix $\delta \in \left(0, \, \frac{1}{2}\right)$.
The interpolation inequality for the Sobolev spaces implicates that for 
any $\ep>0$ there exists a constant $C_{\ep,\delta} > 0$ such that the following inequality holds true (see, e.g., Chapter IV in \cite{Ad} or Sections 2.5 and 11 of Chapter 1 in \cite{LM}):
$$
\Vert v\Vert_{H^{\delta+\frac{1}{2}}(\OOO)} 
\le \ep\Vert \nabla v\Vert_{L^2(\OOO)} + C_{\ep,\delta}\Vert v\Vert_{L^2(\OOO)}
\quad \mbox{for all $v \in H^1(\OOO)$}.
$$
Therefore, we obtain the estimate
\begin{align*}
& \Vert v\Vert^2_{L^2(\ppp\OOO)}
\le (\ep C_{\delta}\Vert \nabla v\Vert_{L^2(\OOO)}
+ C_{\delta}C_{\ep,\delta}\Vert v\Vert_{L^2(\OOO)})^2\\
\le& 2(\ep C_{\delta})^2\Vert \nabla v\Vert_{L^2(\OOO)}^2
+ 2(C_{\delta}C_{\ep,\delta})^2\Vert v\Vert_{L^2(\OOO)}^2
\end{align*}
for all $v \in H^1(\OOO)$.
Substituting this inequality into \eqref{(2)}, we obtain
\begin{align*}
&(\kappa - 2C(\ep C_{\delta})^2) \Vert \nabla v\Vert_{L^2(\OOO)}^2\\
+& ( M - \frac{1}{2}\Vert \mbox{div}\, b\Vert_{C(\ooo{\OOO}\times [0,T])}
- 2(CC_{\delta}C_{\ep,\delta})^2)\Vert v\Vert_{L^2(\OOO)}^2
\le (A_1(t)v,v).
\end{align*}
Choosing a sufficiently small $\ep > 0$  such that 
$\kappa - 2C(\ep C_{\delta})^2> 0$ and a sufficiently large $M>0$ 
such that 
$$
M > \frac{1}{2}\Vert \mbox{div}\, b\Vert_{C(\ooo{\OOO}\times [0,T])}
+ 2(CC_{\delta}C_{\ep,\delta})^2
$$
completes the proof of Lemma \ref{lem1}.
\end{proof}

\vspace{0.3cm}

\noindent
II. Second part of the proof.

\vspace{0.3cm}

Due to the estimate \eqref{1}, we can apply Theorem 3.2 (p. 137) in \cite{LU} that implicates 
existence of a constant $\theta \in (0,1)$ such that for each 
$t \in [0,T]$, a solution $\psi(\cdot,t) \in C^{2+\theta}(\ooo{\OOO})$ to the problem 
\eqref{(4.3)} exists uniquely, where $C^{2+\theta}(\ooo{\OOO})$ is the Schauder space 
defined in the proof of Lemma \ref{l6.1} in Section \ref{sec4}.

Now we introduce an auxiliary function 
\begin{equation}\label{(8.3a)}
\eta(t):= \Vert \psi(\cdot,t)\Vert_{C^{2+\theta}(\ooo{\OOO})}, \quad
0\le t \le T.
\end{equation}
Because of the inclusion $\psi(\cdot,t) \in C^{2+\theta}(\ooo{\OOO})$,
the value of the function  $\eta(t)$ is finite for each $t \in [0,T]$.


Furthermore, for an arbitrary $G \in \HOLSZ$, 
there exists a unique solution $w=w(\cdot,t)$ to the problem
\begin{equation}\label{(3)}
\left\{ \begin{array}{rl}
& A_1(t)w = G \quad \mbox{in $\OOO$}, \\
& \ppp_{\nu_A}w + \sigma w = 0 \quad \mbox{on $\ppp\OOO$}
\end{array}\right.
\end{equation}
for each $t \in [0,T]$.

Now we prove that 
for each $t\in [0,T]$ there exists a constant $C_t > 0$ such that 
\begin{equation}\label{(4)}
\Vert w(\cdot,t)\Vert_{\HOLST} \le C_t\Vert G\Vert_{\HOLSZ}
\end{equation}
for all solutions $w$ of the problem \eqref{(3)}.  In the inequality \eqref{(4)},  the constant $C_t>0$ depends on the norms 
$\Vert a_{ij}\Vert_{C^1(\ooo{\OOO})}$, $1\le i,j\le d$,
$\Vert b_k\Vert_{C([0,T];C^2(\ooo{\OOO}))}$, $0\le k\le d$ of the coefficients, but not on the coefficients by themselves.

Indeed, for each $t \in [0,T]$, the inequality
\begin{equation}\label{(5)}
  \Vert w(\cdot,t)\Vert_{\HOLST} \le C_t(\Vert G\Vert_{\HOLSZ}
+ \Vert w(\cdot,t)\Vert_{C(\ooo{\OOO})})
\end{equation}
holds true (see, e.g., the formula (3.7) on p. 137 in \cite{LU}).
To obtain the desired estimate  we have to eliminate the term  
$\Vert w(\cdot,t)\Vert_{C(\ooo{\OOO})}$
on the right-hand side of the last inequality.
This can be done by the standard compactness-uniqueness arguments.  
More precisely,
let us assume that \eqref{(4)} does not hold.  Then there exist the sequences $w_n\in \HOLST,\ n\in \N$ and  $G_n\in \HOLSZ,\ n\in \N$ such that $\Vert w_n\Vert_{\HOLST} = 1$ and 
$\lim_{n\to\infty}\Vert G_n\Vert_{\HOLSZ} = 0$.
By the Ascoli-Arzel\`a theorem,  we can extract  a subsequence $w_{k(n)}$ 
from the sequence $w_n$ such that 
$w_{k(n)} \longrightarrow \www{w}$ in $C(\ooo{\OOO})$ as $n\to \infty$.
Applying the estimation \eqref{(5)} to the equation
$$
A_1(t)(w_{k(n)} - w_{k(m)}) = G_{k(n)} - G_{k(m)} \quad
\mbox{in $\OOO$}
$$
equipped with the homogeneous boundary condition  $\ppp_{\nu_A}(w_{k(n)} - w_{k(m)}) + 
\sigma(w_{k(n)} - w_{k(m)}) = 0$ on $\ppp\OOO$, we we arrive at the relation 
\begin{align*}
&  \Vert w_{k(n)} - w_{k(m)}\Vert_{\HOLST}\\
\le & C_t(\Vert G_{k(n)} - G_{k(m)}\Vert_{\HOLSZ}
+ \Vert w_{k(n)} - w_{k(m)}\Vert_{C(\ooo{\OOO})})
\,\to 0
\end{align*}
as $n,m \to \infty$.  Hence, there exists a function $w_0 \in \HOLST$ such that 
$w_{k(n)} \to w_0$ in $\HOLST$. Moreover, we obtain the relations
$$
\Vert w_0\Vert_{\HOLST} = \lim_{n\to\infty} \Vert w_{k(n)}\Vert_{\HOLST}
= 1
$$
and $G_{k(n)} = A_1(t)w_{k(n)} \to A_1(t)w_0$ in 
$\HOLSZ$.

Since  $\lim_{n\to\infty} \Vert G_{k(n)}\Vert_{\HOLSZ} = 0$, we obtain
$A_1(t)w_0 = 0$ in $\OOO$ with $\ppp_{\nu_A}w_0 + \sigma w_0 = 0$ on 
$\ppp\OOO$.  Then Lemma \ref{lem1} yields $w_0(x,t)=0$ in $\OOO$ that 
contradicts the relation $\Vert w_0\Vert_{\HOLST} = 1$ that we established above. The obtained contradiction implicates the desired norm estimate 
\eqref{(4)}.

\vspace{0.3cm}

\noindent
III. Third part of the proof.

\vspace{0.3cm}

The last missing detail of the proof is  the inclusion $\psi \in C^1([0,T];C^2(\ooo{\OOO}))$ for the function $\psi$ constructed in the previous part of the proof. 

To show this inclusion, we first verify that for an arbitrarily but fixed $t\in [0,T]$ the function
$d(x,s):= \psi(x,t) - \psi(x,s)$   satisfies the equations
\begin{equation}\label{(6)}
\left\{ \begin{array}{rl}
& -A_1(t)d(\cdot,s) = (b_0(t) - b_0(s))\psi(\cdot,s)\\
- & \sum_{j=1}^d (b_j(t) - b_j(s))\ppp_j\psi(\cdot,s) \quad 
\mbox{in $\OOO$ for $0\le s, t \le T$},\\
& \ppp_{\nu_A}d + \sigma d = 0 \quad \mbox{on $\ppp\OOO$, $\,\,$ 
$0 \le s,t \le T$}.
\end{array}\right.
\end{equation}

For an arbitrarily but fixed $\delta>0$ we set 
$I_{\delta,t}:= [0,T] \cap \{s;\, \vert t-s\vert \le \delta\}$.

Applying the relation  \eqref{(8.3a)} and the estimate \eqref{(4)} to the solution $d$ of \eqref{(6)} yields 
\begin{align*}
& \Vert d(\cdot,s)\Vert_{\HOLST}\\
\le & C\left(\left\Vert \sum_{j=1}^d (b_j(t)-b_j(s))\ppp_j\psi(\cdot,s)
\right\Vert_{\HOLSZ}
+ \Vert (b_0(t)-b_0(s))\psi(\cdot,s)\Vert_{\HOLSZ}\right)\\
\end{align*}
\begin{equation}\label{(8.8a)}
\le C\sum_{j=0}^d \Vert b_j(s) - b_j(t)\Vert_{\CONE}\eta(s)
\le C\max_{0\le j \le d} \Vert b_j(s) - b_j(t)\Vert_{\CONE}\,
\sup_{s \in I_{\delta,t}}\eta(s).
\end{equation}

For the function
$$
h(\delta):= \max_{0\le j\le d} \sup_{\vert s-t\vert \le \delta}
\Vert b_j(s) - b_j(t)\Vert_{\CONE},
$$
the inclusions $b_j \in C([0,T];C^1(\ooo{\OOO}))$, $0\le j \le d$ imply
 $\lim_{\delta\downarrow 0} h(\delta) = 0$.

Now we rewrite the estimate \eqref{(8.8a)} in terms of the function $\eta$ defined by \eqref{(8.3a)} as 
$$
\vert \eta(s) - \eta(t)\vert \le Ch(\delta)\sup_{s\in I_{\delta,t}} \eta(s)
\quad \mbox{for $s \in I_{\delta,t}$},
$$
and thus obtain the inequality
$$
\eta(s) \le \eta(t) + Ch(\delta)\sup_{s\in I_{\delta,t}} \eta(s)
\quad \mbox{for $s \in I_{\delta,t}$}.
$$
Choosing $\delta: =\delta(t)>0$ sufficiently small,  
for a given $t \in [0,T]$, the estimate
$\sup_{s\in I_{\delta(t),t}} \eta(s) \le C_1\eta(t)$ holds true.
Varying $t \in [0,T]$, we now choose a finite number of the  intervals 
$I_{\delta(t),t}$ that cover the whole interval $[0,T]$ and thus  obtain the norm estimate 
\begin{equation}\label{(8.8)}
\Vert \psi\Vert_{L^{\infty}(0,T;C^{2+\theta}(\ooo{\OOO}))} \le 
C_2                             
\end{equation}
with some constant $C_2>0$.

For $s \in I_{\delta(t),t}$, substitution of \eqref{(8.8)} into \eqref{(8.8a)} yields the estimate
$$
\Vert d(\cdot,s)\Vert_{C^{2+\theta}(\ooo{\OOO}))}
= \Vert \psi(\cdot,t) - \psi(\cdot,s)\Vert_{C^{2+\theta}(\ooo{\OOO}))}
\le Ch(\delta)C_2.
$$
Consequently, 
$\lim_{s\to t} \Vert \psi(\cdot,s) - \psi(\cdot,t)\Vert
_{\HOLST} = 0$ and we have shown the inclusion
\begin{equation}\label{(8)}
\psi \in C([0,T];\HOLST).
\end{equation}

To complete the proof, the now verify the inclusion $\psi \in C^1([0,T];\HOLST)$. 
Since $-A_1(\xi)\psi(x,\xi) = 1$ in $\OOO$, 
differentiating this formula with respect to $\xi$
leads to the representation
\begin{equation}\label{(9)}
\sum_{j=1}^d \ppp_i(a_{ij}(x)\ppp_j\ppp_{\xi}\psi(x,\xi))
+ \sum_{j=1}^d b_j(\xi)\ppp_j\ppp_{\xi}\psi(x,\xi)
- b_0(\xi)\ppp_{\xi}\psi(x,\xi)
\end{equation}
$$
= -\sum_{j=1}^d \ppp_{\xi}b_j(\xi)\ppp_j\psi(x,\xi)
+ (\ppp_{\xi}b_0)(\xi)\psi(x,\xi) \quad \mbox{in $\OOO$}.
$$
By subtracting the equation \eqref{(9)} with 
$\xi=s$ from the one with $\xi=t$, we deduce that the function
$d_1(x,s):= (\ppp_t\psi)(x,t) - (\ppp_t\psi)(x,s)$ satisfies the relation
\begin{equation}\label{(10)}
-A_1(t)d_1(x,s)
\end{equation}
\begin{align*}
=& \biggl[ -\sum_{i,j=1}^d (b_j(t)-b_j(s))\ppp_j\ppp_s\psi(x,s)
+ (b_0(t)-b_0(s))\ppp_s\psi(x,s)\\
- & \sum_{j=1}^d (\ppp_tb_j(t)-\ppp_sb_j(s))\ppp_j\psi(x,s)
+ (\ppp_tb_0(t)-\ppp_sb_0(s))\psi(x,s)\biggr]  \\
+ &\left[ -\sum_{j=1}^d \ppp_tb_j(t)(\ppp_j\psi(x,t) - \ppp_j\psi(x,s))
 + \ppp_tb_0(t)(\psi(x,t) - \psi(x,s))\right]\\
=:& H_1(x,t,s) + H_2(x,t,s) \quad \mbox{in $\OOO$}
\end{align*}
and the boundary condition $\ppp_{\nu_A}d_1 + \sigma d_1 = 0$ on $\ppp\OOO$ for all 
$s,t \in [0,T]$.
Thus,  the inclusion 
$\psi\in C^1([0,T];\HOLST)$ will follow from the relation $\lim_{s\to t} \Vert d_1(\cdot,s)\Vert
_{\HOLST} = 0$ that we now prove. 
To this end, by applying Theorem 3.2 (p. 137) in \cite{LU} to the equation \eqref{(10)},
it suffices to prove that 
\begin{equation}\label{(11)}
\lim_{s\to t} \Vert H_{\ell}(\cdot,t,s)\Vert_{\HOLSZ} = 0, \quad
\ell=1,2. 
\end{equation}

Applying Theorem 3.2 in \cite{LU} to the equation \eqref{(9)}, in view of the 
regularity conditions \eqref{(1.2)} and the equation \eqref{(9)}, we obtain
the estimates
\begin{equation}\label{(12)}
\Vert \ppp_t\psi(\cdot,t)\Vert_{\HOLST}
\end{equation}
\begin{align*}
\le& C\left( \left\Vert \sum_{j=1}^d (\ppp_tb_j)(\cdot,t)
\ppp_j\psi(\cdot,t)\right\Vert_{\HOLSZ}
+ \Vert (\ppp_tb_0)(\cdot,t)\psi(\cdot,t)\Vert_{\HOLSZ}\right)\\
\le& C\sum_{j=0}^d \Vert b_j(\cdot,t)\Vert_{C^1([0,T];\CONE)}
\Vert \psi(\cdot,t)\Vert_{\HOLST}
\le C_3 \quad \mbox{for $0\le t\le T$.}
\end{align*}
The inequalities \eqref{(8.8)} and \eqref{(12)} lead to the norm estimate
\begin{equation}\label{(13)}
\Vert H_1(\cdot,t,s)\Vert_{\HOLSZ}
\le C_4\sum_{k=0}^1\sum_{j=0}^d
\Vert (\ppp_t^kb_j)(\cdot,t) - (\ppp_t^kb_j)(\cdot,s)\Vert
_{\CONE}.
\end{equation}
By employing the arguments similar to those used above, the estimate
\begin{equation}\label{(13a)}
\Vert H_2(\cdot,t,s)\Vert_{\HOLSZ}
\le C_4\sum_{j=0}^d \Vert \ppp_tb_j\Vert_{C([0,T];\CONE)}
\sum_{k=0}^1 \Vert \nabla^k\psi(\cdot,t) - \nabla^k\psi(\cdot,s)
\Vert_{\HOLSZ}
\end{equation}
can be derived. 
Since $\ppp_t^kb_j \in C([0,T];C^1(\ooo{\OOO}))$ for $k=0,1$ and
$0\le j \le d$ by the conditions \eqref{(1.2)} and $\nabla^k\psi \in C([0,T];C^{1+\theta}
(\ooo{\OOO}))$ with $k=0,1$ by the inclusion \eqref{(8)}, the relation $\lim_{s\to t} \Vert H_{\ell}(\cdot,t,s)\Vert_{\HOLSZ} = 0$ with
$\ell=1,2$ immediately follows from the norm estimates \eqref{(13)} and \eqref{(13a)}. As mentioned above, this completes the proof 
 of existence of a function  satisfying the conditions \eqref{(4.3)}. 




\begin{acknowledgements}
The second author was supported by Grant-in-Aid 
for Scientific Research Grant-in-Aid (A) 20H00117 of 
Japan Society for the Promotion of Science.
\end{acknowledgements}

 \section*{\small
 Conflict of interest} 

 {\small
 The authors declare that they have no conflict of interest.}




\bigskip  

\small 
\noindent
{\bf Publisher's Note}
Springer Nature remains neutral with regard to jurisdictional claims in published maps and institutional affiliations.

\end{document}